\documentclass[12pt,reqno]{amsart}
\usepackage[headings]{fullpage}
\usepackage{amssymb,amsmath,amscd,bbm,tikz-cd}
\usepackage[all,cmtip]{xy}
\usepackage{url}
\usepackage[bookmarks=true,%
colorlinks=true,%
linkcolor=blue,
citecolor=blue,%
filecolor=blue,%
menucolor=blue,%
urlcolor=blue,%
breaklinks=true]{hyperref}
\usepackage{slashed}
\usepackage{listings}
\usepackage{verbatim}
\usepackage{mathtools}
\usepackage[normalem]{ulem}   
\usepackage{graphicx}
\usepackage{texdraw}
\usepackage{caption}
\graphicspath{{figures/}}

\newtheorem{theorem}{Theorem}[section]

\newtheorem{conjecture}[theorem]{Conjecture}

\newtheorem{lemma}[theorem]{Lemma}

\theoremstyle{definition}
\newtheorem{definition}[theorem]{Definition}

\newtheorem{example}[theorem]{Example}
\newtheorem{remark}[theorem]{Remark}

\newtheorem{claim}{Claim}

\newcommand{\rvline}{\hspace*{-\arraycolsep}\vline\hspace*{-\arraycolsep}}

\def\be{\begin{equation}}
\def\ee{\end{equation}}

\def\BZ{\mathbbm Z}
\def\BQ{\mathbbm Q}

\def\BC{\mathbbm C}

\def\BF{\mathbbm F}

\def\calT{\mathcal T}

\def\calX{\mathcal X}
\def\calY{\mathcal Y}
\def\calZ{\mathcal Z}

\def\geom{\mathrm{geom}}

\def\SL{\mathrm{SL}_2(\mathbb{C})}
\def\PSL{\mathrm{PSL}_2(\mathbb{C})}
\def\sl{\mathfrak{sl}_2(\mathbb{C})}
\def\SLC{\SL \!\ltimes \BC^2}
\def\tr{\mathrm{tr} \,}
\def\a{\alpha}

\newcommand{\maps}{\colon\thinspace}
\newcommand{\reduced}{\mathrm{red}}
\DeclareMathOperator{\Gal}{Gal}
\DeclareMathOperator{\vol}{vol}
\newcommand{\Qbar}{\overline{\BQ}}
\newcommand{\rhobar}{\overline{\rho}}

\tikzset{%
  nmdstd/.style={%
    line join=round,
    line cap=round,
    font=\footnotesize,
    >={Computer Modern Rightarrow[length=1pt 5, width'=0pt 1]},
  }
}

\makeatletter
\newenvironment{tikzoverlay*}[2][]{%
  \node[anchor=south west, inner sep=0] (image) at (0,0) {\includegraphics[#1]{#2}};
  \newdimen\nmd@tikzoverlaywidth
  \pgfextractx{\nmd@tikzoverlaywidth}{\pgfpointanchor{image}{south east}}
  \begin{scope}
    \tikzset{x=0.01\nmd@tikzoverlaywidth}
    \tikzset{y=0.01\nmd@tikzoverlaywidth}
  }{
  \end{scope}
}
\makeatother

\makeatletter

\renewcommand\thepart{\@Roman\c@part}%
\renewcommand\part{%
   \if@noskipsec \leavevmode \fi
   \par
   \addvspace{6.7ex}%
   \@afterindentfalse
   \secdef\@part\@spart}
\def\@part[#1]#2{%
    \ifnum \c@secnumdepth >\m@ne
      \refstepcounter{part}%
      \addcontentsline{toc}{part}{Part~\thepart.\ #1}%
    \else
      \addcontentsline{toc}{part}{#1}%
    \fi
    {\parindent \z@ \raggedright
     \interlinepenalty \@M
     \normalfont
     \ifnum \c@secnumdepth >\m@ne
       \centering\large\scshape \partname~\thepart.%
       \hspace{1ex}%
     \fi%
     \large\scshape #2%
     \markboth{}{}\par}%
    \nobreak
    \vskip 4.7ex
    \@afterheading}
  \def\@spart#1{
  \refstepcounter{part}%
  \addcontentsline{toc}{part}{#1}%
    {\parindent \z@ \raggedright
     \interlinepenalty \@M
     \normalfont
     \centering\large\scshape #1\par}%
     \nobreak
     \vskip 4.7ex
     \@afterheading}
\renewcommand*\l@part[2]{%
  \ifnum \c@tocdepth >-2\relax
    \addpenalty\@secpenalty
    \addvspace{0.75em \@plus\p@}%
    \begingroup
      \parindent \z@ \rightskip \@pnumwidth
      \parfillskip -\@pnumwidth
      {\leavevmode
       \normalsize \bfseries #1\hfil \hb@xt@\@pnumwidth{\hss #2}}\par
       \nobreak
       \if@compatibility
         \global\@nobreaktrue
         \everypar{\global\@nobreakfalse\everypar{}}%
      \fi
    \endgroup
  \fi}

\def\l@subsection{\@tocline{2}{0pt}{2pc}{6pc}{}}
\makeatother

\definecolor{codegreen}{rgb}{0,0.6,0}
\definecolor{codegray}{rgb}{0.5,0.5,0.5}
\definecolor{codepurple}{rgb}{0.58,0,0.82}
\definecolor{backcolour}{rgb}{0.95,0.95,0.92}

\lstdefinelanguage{PARIGP}{
  keywords={typeof, new, true, false, catch, function, return, null, catch,
    switch, var, if, in, while, do, else, case, break},
  keywordstyle=\color{blue}\bfseries,
  ndkeywords={class, export, boolean, throw, implements, import, this},
  ndkeywordstyle=\color{darkgray}\bfseries,
  identifierstyle=\color{black},
  sensitive=false,
  comment=[l]{//},
  morecomment=[s]{/*}{*/},
  commentstyle=\color{purple}\ttfamily,
  stringstyle=\color{red}\ttfamily,
  morestring=[b]',
  morestring=[b]"
}

\lstdefinestyle{code}{
    backgroundcolor=\color{backcolour},   
    commentstyle=\color{codegreen},
    keywordstyle=\color{magenta},
    numberstyle=\tiny\color{codegray},
    stringstyle=\color{codepurple},
    basicstyle=\ttfamily\footnotesize,
    breakatwhitespace=false,         
    breaklines=true,                 
    captionpos=b,                    
    keepspaces=true,                 
    numbers=left,                    
    numbersep=5pt,                  
    showspaces=false,                
    showstringspaces=false,
    showtabs=false,                  
    tabsize=2
}

\begin{document}
	
\title[1-loop equals torsion for fibered 3-manifolds]{
  1-loop equals torsion for fibered 3-manifolds}

\author{Nathan M. Dunfield}
\address{University of Illinois Urbana-Champaign, Dept.~of Mathematics\\
  Urbana, IL 61801, USA}
\email{nathan@dunfield.info}
\urladdr{https://dunfield.info}

\author{Stavros Garoufalidis}
\address{
	International Center for Mathematics, Department of Mathematics \\
	Southern University of Science and Technology \\
	Shenzhen, China}
\email{stavros@mpim-bonn.mpg.de}
\urladdr{http://people.mpim-bonn.mpg.de/stavros}

\author{Seokbeom Yoon}
\address{Department of Mathematics \\
	Chonnam National University \\
	Gwangju, South Korea} 
\email{sbyoon15@gmail.com}
\urladdr{http://sites.google.com/view/seokbeom}



\date{31 March 2024}
		
\begin{abstract}
In earlier work of two of the authors, two 1-loop polynomial
invariants of cusped 3-manifolds were constructed using combinatorial
data of ideal triangulations, and conjectured to be equal to the
$\BC^2$ and the $\BC^3$-torsion polynomials. Here, we prove this
conjecture for layered triangulations of fibered 3-manifolds with
toroidal boundary, and we illustrate our theorems with exact
computations of the 1-loop and the torsion polynomials.  As further
evidence for the conjecture, we confirm it for more than 6,600
nonfibered manifolds, and use this data to explore the extent to which
the $\BC^2$-torsion determines the Thurston norm.
\end{abstract} 
	
\maketitle
{

\vspace{-0.5cm}
  
\tableofcontents
}


\section{Introduction}
\label{sec.intro}

The torsion polynomials of a hyperbolic 3-manifold are powerful
invariants defined from the geometric representation that encodes its
hyperbolic structure. They can be used to certify minimizing
properties of surfaces representing a fixed homology class in the
manifold (i.e., the Thurston norm), highlighting a beautiful and
interesting connection between the geometry and topology in dimension
three.

Here and throughout the paper, a 3-manifold means a compact orientable
3-manifold with non-empty toroidal boundary. (We do not distinguish
such a 3-manifold from its interior.)  The $\BC^n$-torsion polynomial
$\tau_{n,\rho}(t)$ depends on a 3-manifold $M$, an integer cohomology
class $\alpha \in H^1(M,\BZ)$ and an $\SL$-representation $\rho$ of
$\pi_1(M)$, and it is defined to be the twisted Alexander polynomial
associated to the $(n-1)$-st symmetric power of $\rho$ and $\alpha$.
More precisely, $\tau_{n,\rho}(t)$ is the twisted torsion polynomial
of this data which is well-defined up to a multiple $\pm t^k$ with
$k \in \BZ$; see for example \cite[\S 2]{Dunfield:twisted}.  One
important feature of the torsion polynomials is that they give a bound
on the Seifert genus of a hyperbolic knot $K \subset S^3$, that is,
the smallest genus of an oriented surface spanning $K$.  Precisely,
for all $n \geq 2$, one has
\begin{equation} \label{eq.genusbound}
  2 \cdot \mathrm{genus}(K) -1
  \geq \frac{1}{n} \deg \tau_{n,\rho^\geom}(t) \,,
\end{equation}
where $\rho^\geom$ is any $\SL$-lift of the geometric representation of the
hyperbolic knot complement $M$ and $\alpha \in H^1(M,\BZ)$ is given by the
abelianization of $\pi_1(M)$. This inequality appears (and in many
cases is shown) to be sharp for $n=2$, but is not always sharp for
$n=3$~\cite{Agol-Dunfield,Dunfield:twisted}.

On the other hand, motivated by Chern--Simons perturbation theory (and
the connection between 1-loop perturbation theory and Ray--Singer
torsion), Dimofte and the first author \cite{DG1} introduced the
1-loop invariant and showed its topological invariance. The 1-loop
invariant is defined from an ideal triangulation of a 3-manifold $M$
and a solution to Thurston's gluing equations, or equivalently, to the
Ptolemy equations.  A generalization of the 1-loop invariant to a
1-loop polynomial $\delta_{3,c}(t)$ was given in~\cite{GY21} by
replacing the gluing equations matrices with a twisted version that
depends on a cohomology class $\alpha \in H^1(M,\BZ)$. Adding
face-variables and face-matrices (which are motivated by
super-hyperbolic geometry) to these ingredients, another 1-loop polynomial
$\delta_{2,c}(t)$ was recently defined in~\cite{GY23}.  Here $c$ is a
Ptolemy assignment, i.e., a solution to the Ptolemy equations, with or
without a boundary obstruction class (see Section~\ref{sec.prelim} for
details). Note that a Ptolemy assignment $c$ gives rise to an
$\SL$-representation $\rho_c$ of $\pi_1(M)$.

Physics principles predict that the torsion and the 1-loop polynomials
should be equal to each other:

\begin{conjecture}
\label{conj.23}
Fix a 3-manifold $M$, an ideal triangulation $\calT$ of $M$ and a cohomology class
$\alpha \in H^1(M,\BZ)$. Then for $n=2,3$ and for
all Ptolemy assignments $c$ on $\calT$ with $\rho_c$ irreducible, we have 
\be
\label{conj23}
\delta_{n,c}(t) \doteq \tau_{n,\rho_c}(t) \,.
\ee
where $\doteq$ means the equality of Laurent polynomials up to a
multiple of $\pm t^k$ for $k \in \BZ$.
\end{conjecture}

Our main result, Theorem~\ref{thm.1} below, proves
Conjecture~\ref{conj.23} for certain triangulations of 3-manifolds
that fiber over the circle.  Such a manifold can be written as the
mapping torus $M_\varphi$ of a homeomorphism
$\varphi: \Sigma \to \Sigma$ of a punctured surface $\Sigma$, and has
a distinguished class $\alpha \in H^1(M,\BZ)$ induced from the
fibration $M_\varphi \rightarrow S^1$.  Moreover, using an ideal
triangulation on $\Sigma$ and representing $\varphi$ by a sequence of
2--2 moves (also known as ``flips''~\cite{Hatcher}), one obtains a
layered triangulation $\calT_\varphi$ of $M_\varphi$. Layered
triangulations were used by Jorgensen, Floyd and Hatcher to study
geometry of once-punctured torus bundles
(see~\cite[App.~A]{Floyd-Hatcher}), and more recently
by~\cite{Lackenby}, who generalizes them to the much larger class of taut
ideal triangulations; the latter also certify the Thurston norm.

\begin{theorem} 
\label{thm.1}
Suppose $M_\varphi$ is a 3-manifold that fibers over the circle with
$\alpha \in H^1(M,\BZ)$ corresponding to the fibration $M_\varphi \to S^1$.
Suppose $\calT_\varphi$ is an associated layered ideal triangulation.
For all Ptolemy assignments $c$ on $\calT_\varphi$ with $\rho_c$
irreducible, we have $\delta_{n,c}(t) \doteq \tau_{n,\rho_c}(t)$ for $n =
2, 3$ as posited in Conjecture~\ref{conj.23}.
\end{theorem}

\noindent
Let us add some comments to the theorem.
\smallskip

\noindent

\noindent
1. When $M_\varphi$ is hyperbolic, every $\SL$-lift of the geometric representation
of $M_\varphi$ is irreducible and is of the form $\rho_c$; thus
Theorem~\ref{thm.1} applies in this situation.
\smallskip

\noindent
2. Theorem~\ref{thm.1} implies the original 1-loop conjecture
of~\cite{DG1} for such $(M, \alpha, \calT_\varphi, \rho_c)$, namely
the equality of the the 1-loop invariant with the $\BC^3$-torsion
(also known as the adjoint Reidemeister torsion). This is because the
derivatives of the 1-loop polynomial $\delta_{3,c}(t)$ and the
$\BC^3$-torsion polynomial $\tau_{3,\rho}(t)$ at $t=1$ are the 1-loop
invariant and the $\BC^3$-torsion, respectively~\cite{Yamaguchi,GY21}.
\smallskip


\subsection{Evidence for the conjecture in the nonfibered case.}

The torsion and 1-loop polynomials can be computed using the
methods of ideal triangulations as developed in
\texttt{SnapPy}~\cite{snappy}, see~\cite{DVN/OK6YGC_2020}. In
Section~\ref{sec.evidence}, we check Conjecture~\ref{conj.23} for more
than 6,600 nonfibered 1-cusped hyperbolic 3-manifolds where the all
the exact Ptolemy assignments are known by work of
Goerner~\cite{Goerner:unhyp}.  Also in Section~\ref{sec.evidence}, we
compute the Thurston norm and fibering status for the more than 59,000
manifolds from Burton's census \cite{Burton:census}. We use this data
to explore whether the analogue of (\ref{eq.genusbound}) is always
sharp for $n=2$ for more general 3-manifolds. Here, we discovered a
new phenomena: whether the inequality is sharp sometimes depends on
which $\SL$-lift of the holonomy representation is chosen.  When
$b_1(M) = 1$ for a 3-manifold $M$, define $x(M) \in \BZ_{\geq 0}$ to
be the Thurston norm of a generator of $H^1(M; \BZ)$; one always has
$x(M) \geq \frac{1}{2} \deg \tau_{\rho^\geom}(t)$ by
\cite{Dunfield:twisted}.  Theorem~\ref{thm.sharpy} below is a more
detailed version of:

\begin{theorem}
  For a sample of more than 6,600 nonfibered 1-cusped hyperbolic
  3-manifolds with $b_1(M) = 1$, there was always a lift $\rho$ of the
  holonomy representation where $x(M) = \frac{1}{2} \deg \tau_{2, \rho}(t)$.  However,
  for 50 of these manifolds there was at least one other lift $\psi$
  where $x(M) > \frac{1}{2} \deg \tau_{2, \psi}(t)$.
\end{theorem}

\subsection{Additional comments}

A natural question is the extension of Conjecture~\ref{conj.23} to all
$n \geq 2$.  This is possible but rather long and technical, and a
sketch of the main ideas involved are given in Section 4
of~\cite{GY23}. With much additional work, Theorem~\ref{thm.1} can be
extended to all $n \geq 2$, but we will not discuss this here, partly
because all known applications to genus detection or to quantum
topology involve only the case of $n=2,3$.

Also, our 1-loop polynomials are defined by determinantal
formulas (see Section~\ref{sub.1loops} below), which in the case of fibered
3-manifolds closely resemble the definition of the taut and the veering
polynomials of Landry--Minsky--Taylor~\cite{Landry:polynomial},
themselves a generalization of the Teich\"muller polynomials of
McMullen~\cite{Mcmullen:polynomial}; see also Parlak~\cite{Parlak}. This is
an intruiging connection, with similarities and differences: 
\begin{itemize}
\item[(a)]
  all polynomials require an element $\alpha$ of $H^1(M,\BZ)$,
\item[(b)]
  the 1-loop polynomials use arbitrary triangulations whereas the taut/veering
  polynomials are defined only for taut/veering triangulations,
\item[(c)]
  the 1-loop polynomials require a solution to the Ptolemy equations whereas the
  taut/veering polynomials do not. 
\end{itemize}
Although such syntactical similarity plays no role in the proof of
Theorem~\ref{thm.1}, it is worth pointing it out.
\smallskip

\subsection{Acknowledgments.} 
Yoon wishes to thank Teruaki Kitano and Hyuk Kim for helpful conversations, 
and Philip Choi and Seonhwa Kim for conducting numerical experiments.
Dunfield was partially supported by US NSF grants
DMS-1811156 and DMS-2303572 and by a fellowship from the Simons
Foundation (673067, Dunfield). 


\section{Preliminaries}
\label{sec.prelim}

\subsection{Ptolemy and face equations}
\label{sub.eqns}

In this section we recall the Ptolemy and face equations associated to ideal
triangulations of 3-manifolds.

Let $\calT$ be a concrete ideal triangulation of a 3-manifold $M$, that is, a
triangulation such that each tetrahedron comes with a bijection of its vertices
with those of the standard 3-simplex. We denote by $\calT^1$ and $\calT^2$ the
oriented edges and the unoriented faces of $\calT$, respectively.

A Ptolemy assignment on $\calT$ is a map $c : \calT^1 \rightarrow \BC^\ast$
satisfying $c(-e)=-c(e)$ for all $e \in \calT^1$ and the Ptolemy equation
\be
\label{eqn.ptolemy}
P_\Delta :  \quad c(e_{01}) c(e_{23}) - c(e_{02}) c(e_{13}) + c(e_{03}) c(e_{12})=0
\ee
for each tetrahedron $\Delta$ of $\calT$.  Here $e_{ij}$ is the oriented edge $[i,j]$ 
of $\Delta = [0,1,2,3]$ as in Figure~\ref{fig.tetrahedron}.

\begin{figure}[htpb!]
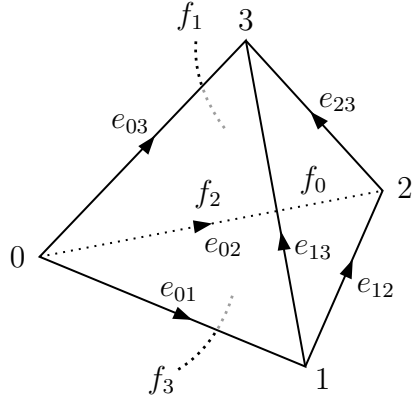
\caption{Edge and face labels for a tetrahedron.}
\label{fig.tetrahedron}
\end{figure}

Since $\calT$ has the same number $N$ of tetrahedra as edges, a Ptolemy assignment
on $\calT$ is represented by an $N$-tuple $(c_1,\ldots,c_N)$ of non-zero complex
numbers satisfying $N$ equations. It follows that the set $P_2(\calT)$ of all Ptolemy
assignments on $\calT$ is given by
\begin{equation*}
P_2(\calT) = \{ (c_1,\ldots, c_N) \in (\BC^\ast)^N \, | \, P_{1} = \cdots = P_{N}=0 \} 
\end{equation*}
where $P_1, \ldots, P_N\in \BZ[c_1,\ldots,c_N]$ are the Ptolemy equations for the
tetrahedra  of $\calT$.

A super-Ptolemy assignment on $\calT$ is a Ptolemy assignment $c :\calT^1
\rightarrow \BC^\ast$ together with a map $\theta : \calT^2 \rightarrow \BC$
satisfying one equation, called a face equation, for each face of each
tetrahedron $\Delta$ of $\calT$:
\be
\label{eqn.face}
\begin{aligned}
E_{\Delta,f_3} & : & c(e_{12}) \theta(f_0) - c(e_{02}) \theta(f_1) + c(e_{01}) \theta(f_2)
& = 0 \\
E_{\Delta,f_2} & : & c(e_{13}) \theta(f_0) - c(e_{03}) \theta(f_1) + c(e_{01}) \theta(f_3)
& = 0 \\ 
E_{\Delta,f_1} & : & c(e_{23}) \theta(f_0) - c(e_{03}) \theta(f_2) + c(e_{02}) \theta(f_3)
& = 0 \\
E_{\Delta,f_0} & : & c(e_{23}) \theta(f_1 )- c(e_{13}) \theta(f_2) + c(e_{12}) \theta(f_3)
& = 0 
\end{aligned}
\ee
Here $e_{ij}$ is the oriented edge $[i,j]$ of $\Delta = [0,1,2,3]$ and $f_k$ is the face 
of $\Delta$ opposite to the vertex $k$ as in Figure~\ref{fig.tetrahedron}. Any three
equations in~\eqref{eqn.face} are linearly dependent; for instance,
\begin{equation*}
c(e_{01})E_{\Delta,f_1} - c(e_{02}) E_{\Delta,f_2} +  c(e_{03}) E_{\Delta,f_3} = 0 \,.
\end{equation*}
It follows that the map $\theta$ is only required to satisfy two equations for each 
tetrahedron, hence total $2N$ equations. Since $\calT$ has $2N$ faces, the set
$P_{2|1}(\calT)$ of all super-Ptolemy assignments on $\calT$ is given by
\begin{equation*}
P_{2|1}(\calT) = \{ (c_1,\ldots, c_N, \theta_1,\ldots,\theta_{2N}) \in (\BC^\ast)^N \times
\BC^{2N} \, | \, P_{1} = \cdots = P_{N}= E_1 = \cdots = E_{2N}=0 \} 
\end{equation*}
where $E_{1},\ldots,E_{2N} \in \BZ[c_1,\ldots,c_N,\theta_1,\ldots,\theta_{2N}]$ consist
of two face equations of each tetrahedron of $\calT$.  Note that a choice of two
faces of a tetrahedron is uniquely determined by their common edge, hence the face
equations $E_1,\ldots,E_{2N}$ are determined by $N$ edges, one in each tetrahedron,
possibly with repetition.

\subsection{Representations}
\label{sub.reps}

In this section we briefly recall how Ptolemy and super-Ptolemy assignments lead to 
representations of $\pi_1(M)$.

A Ptolemy assignment $c$ on $\calT$ determines an $\SL$-matrix to every
edge of the truncated triangulation of $\calT$ as in Figure~\ref{fig.p2c}. It is proved
in~\cite{GTZ15} that these matrices form a 1-cocycle, that is, a representation of
the groupoid generated by the 2-skeleton of the truncated triangulation, and 
hence determine a representation $\rho_c : \pi_1(M)\rightarrow\SL$, well-defined up
to conjugation. This defines a map
\be
\label{eqn.rep1}
P_2(\calT) \rightarrow \calX(M,\SL), \quad c \mapsto \rho_c
\ee
where
\be
\label{calXM}
\calX(M,\SL)=\mathrm{Hom}(\pi_1(M),\SL)/\! \!/\SL
\ee
is the $\SL$-character variety of $M$. Note that every peripheral curve $\gamma$ 
of $M$ satisfies that $\mathrm{tr}(\rho_c(\gamma))=2$.

\begin{figure}[htpb!]
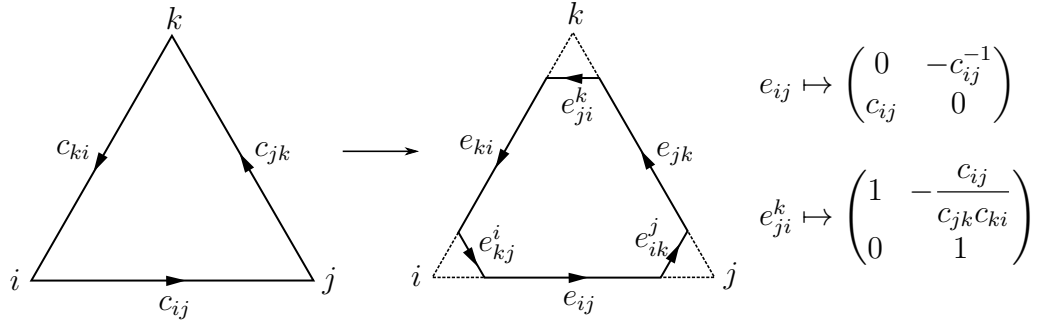
\caption{From Ptolemy assignments to $\SL$-assignments.}
\label{fig.p2c}
\end{figure}

The map $\theta$ of a super-Ptolemy assignment $(c,\theta)$ on $\calT$ 
allows us to additionally assign a $\BC^2$-vector to every edge of the truncated
triangulation as in Figure~\ref{fig.p2c2}. These $\BC^2$-vectors with
$\SL$-matrices given in Figure~\ref{fig.p2c} form special affine transformations 
of $\BC^2$.  It is proved in \cite{GY23} that these affine transformations form a
1-cocycle, hence (up to conjugation) we obtain a representation 
$\rho_{(c,\theta)} : \pi_1(M) \rightarrow\SLC$. This defines a map
\be
\label{eqn.rep2}
P_{2|1}(\calT) \rightarrow \calX(M,\SLC), \quad (c,\theta) \mapsto \rho_{(c,\theta)}
\ee
where $\calX(M,\SLC)$ is the $(\SLC)$-character variety of $M$.

\begin{figure}[htpb!]
\input{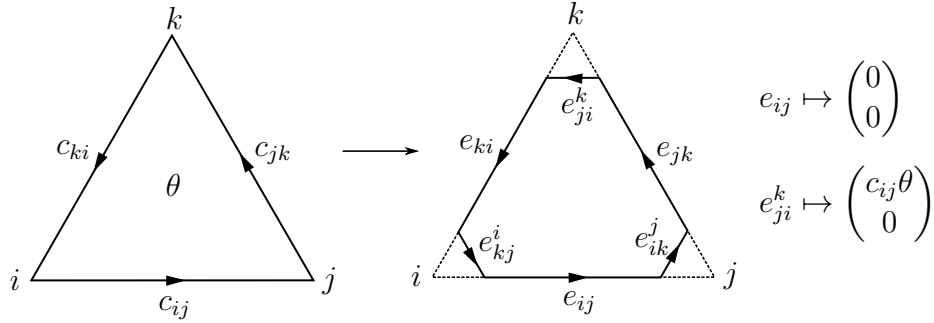}
\caption{From super-Ptolemy assignments to $\BC^2$-assignments.}
\label{fig.p2c2}
\end{figure}

The set $P_2(\calT)$ admits a $(\BC^\ast)^b$-action, where $b$ is the number of
boundary components of $M$, defined as follows.
\be
\label{eqn.act}
(\BC^\ast)^b \times P_{2}(\calT) \rightarrow P_{2}(\calT), \quad
(x, c) \mapsto x \cdot c
\ee 
Regarding $x =(x_1,\ldots,x_b)$ is assigned to the vertices of $\calT$, 
$(x \cdot c)(e) := x_i x_j c(e)$ for $e \in \calT^1$ where $x_i$ and $x_j$ are
assigned to the vertices of $e$. This action preserves associated
$\SL$-representations (see \cite{GTZ15, Garoufalidis:ptolemy}).
Namely, the map~\eqref{eqn.rep1} induces
\be
\label{eqn.rep3}
P_2(\calT)/(\BC^\ast)^b \rightarrow \calX(M,\SL)\,.
\ee
Although we do not need this in our paper, we remark that the above action extends
to $P_{2|1}(\calT)$~\cite{GY23}.

\begin{remark}
\label{rmk.surface}
Ptolemy and super-Ptolemy assignments also work for punctured surfaces.
However, Ptolemy or face equations will not appear, as there is no tetrahedron.
Note that an ideal triangulation $\calT$ of a punctured surface $\Sigma$ with Euler
characteristic $-n<0$ has $3n$ edges and $2n$ faces. It follows that we have
bijections
\be
\label{Sigmaiso}
P_2(\calT) \simeq (\BC^\ast)^{3n}, \qquad 
P_{2|1}(\calT) \simeq (\BC^\ast)^{3n} \times \BC^{2n} \,.
\ee
The maps~\eqref{eqn.rep1} and~\eqref{eqn.rep2} become
\be
(\BC^\ast)^{3n} \rightarrow  \calX(\Sigma,\SL) \quad \textrm{and}
\quad  (\BC^\ast)^{3n} \times \BC^{2n} \rightarrow \calX(\Sigma,\SLC)
\ee
respectively, and the former map induces a map $(\BC^\ast)^{3n}/(\BC^\ast)^p \rightarrow 
\calX(\Sigma,\SL)$ as in~\eqref{eqn.rep3}, where $p$ is the number of punctures of
 $\Sigma$.
\end{remark}

\subsection{Polynomials}
\label{sub.1loops}

In this section we recall the definition of the 1-loop polynomials $\delta_{3,c}(t)$ and
$\delta_{2,c}(t)$ associated to a Ptolemy assignment $c \in P_2(\calT)$,
well-defined up to multiplying signs and monomials in $t$~\cite{GY21,Yoon21, GY23}.
To do so,  we fix a cohomology class $\alpha \in H^1(M, \BZ)$ and denote by 
$\widetilde{M}$ the infinite cyclic cover of $M$ determined by $\alpha$ and by
$\widetilde{\calT}$ the ideal triangulation of $\widetilde{M}$ induced from~$\calT$.

We fix a lift $\widetilde{\Delta}_j$ of each tetrahedron $\Delta_j$ of $\calT$ and
denote the edges of $\widetilde{\Delta}_1,\ldots, \widetilde{\Delta}_N$ in
$\widetilde{\calT}$ by $\widetilde{e}_{1}, \ldots, \widetilde{e}_{I}$. The number $I$
of the edges is greater than $N$, and  some of them project down to the same edge
in $\calT$. Precisely, there are $I-N$ independent relations of the form
\begin{equation} \label{eqn.UniqueLift}
\widetilde{e}_{i} = t^{k} \cdot \widetilde{e}_{j}
\end{equation}
where $t$ is a generator of the deck transformation group $\BZ$ of $\widetilde{M}$.

Assigning a variable $\widetilde{c}_i$ to each edge $\widetilde{e}_i$ for
$1 \leq i \leq I$, we obtain $N$ elements
\be 
\widetilde{P}_1 , \ldots, \widetilde{P}_N
\in \BZ[\widetilde{c}_1,\ldots,\widetilde{c}_I]  
\ee
from the Ptolemy equations of $\widetilde{\Delta}_1,\ldots,\widetilde{\Delta}_N$.
Furthermore, since the relation~\eqref{eqn.UniqueLift} naturally corresponds to an
element $\widetilde{c}_{i} - t^{k} \, \widetilde{c}_{j}$, we obtain additional $I-N$
elements 
\be
\widetilde{Q}_1,\ldots,\widetilde{Q}_{I-N}
\in \BZ[\widetilde{c}_1 ,\ldots,\widetilde{c}_{I}][t^{\pm1}]
\ee
from the relations on the edges of $\widetilde{\Delta}_1,\ldots,\widetilde{\Delta}_N$. 
Note that a Ptolemy assignment on $\calT$ corresponds in an obvious way to an
$I$-tuple $(\widetilde{c}_1,\ldots, \widetilde{c}_I)$ of non-zero complex numbers
satisfying $\widetilde{P}_1 = \cdots = \widetilde{P}_N  = 0$ and $\widetilde{Q}_1 =
\cdots = \widetilde{Q}_{I-N}=0$ at $t=1$.  Given $\widetilde{F}_1,
\ldots, \widetilde{F}_L$ in
$\BZ[\widetilde{c}_1,\ldots,\widetilde{c}_I]$, the Jacobian matrix
\be
\dfrac{ \partial
    (\widetilde{F}_1 ,\ldots, \widetilde{F}_{L}) }
  { \partial (\widetilde{c}_1 ,\ldots, \widetilde{c}_I)}
  \quad \mbox{has $(i, j)$-entry $\dfrac{\partial
      \widetilde{F}_i}{\partial \, \widetilde{c}_j}.$}
\ee

\begin{definition} 
Associated to a Ptolemy assignment $c \in P_2(\calT)$, we define
\be
\label{def.3}
\delta_{3,c}(t) \doteq
 \left(\prod_{e }  \frac{1}{c(e)} \right) \det \left( \dfrac{ \partial
    (\widetilde{Q}_1 ,\ldots \widetilde{Q}_{I-N}, \widetilde{P}_1, \ldots,
    \widetilde{P}_N) }
{ \partial (\widetilde{c}_1 ,\ldots, \widetilde{c}_I)} \right)
\ee	
where the product is taken over all edges $e$ of $\calT$.
\end{definition}

Note that the first $I-N$ rows of the Jacobian matrix in~\eqref{def.3} have
precisely two non-zero entries, one of which is $1$. Therefore, using elementary
row operations, one can reduce the size of the Jacobian matrix to $N$ whose entries
are given by $\BZ[t^{\pm1}]$-linear combinations of $c(e)$.

Similarly, we assign a variable $\widetilde{\theta}_j$ to each face of
$\widetilde{\Delta}_1,\ldots, \widetilde{\Delta}_N$ in $\widetilde{\calT}$.  
The number $J$ of  the faces of $\widetilde{\Delta}_1,\ldots, \widetilde{\Delta}_N$
in $\widetilde{\calT}$ is greater than $2N$, and there are $J-2N$ independent
relations on those faces (as in \eqref{eqn.UniqueLift}) which define $J-2N$ elements
\be
\widetilde{F}_1,\ldots,\widetilde{F}_{J-2N}
\in \BZ[\widetilde{\theta}_1,\ldots,\widetilde{\theta}_J][t^{\pm1}]
\ee
of the form  $\widetilde{\theta}_{i} - t^{k} \, \widetilde{\theta}_{j}$. 
Also, we have
\be 
\widetilde{E}_1 , \ldots, \widetilde{E}_{2N}
\in \BZ[\widetilde{c}_1,\ldots,\widetilde{c}_I][
\widetilde{\theta}_1,\ldots,\widetilde{\theta}_J]
\ee
from face equations of $\widetilde{\Delta}_1,\ldots,\widetilde{\Delta}_N$, where
each tetrahedron contributes two face equations.

\begin{definition}
Associated to a Ptolemy assignment $c \in P_2(\calT)$, we define
\be 
\label{eqn.def2}
\delta_{2,c}(t) \doteq
 \left(\prod_{e}  \frac{1}{c(e)} \right) \left(\prod_{\Delta}
  \frac{1}{c(e_\Delta)} \right)\det \left( \dfrac{ \partial
    (\widetilde{F}_1 ,\ldots \widetilde{F}_{J-2N},
    \widetilde{E}_1, \ldots, \widetilde{E}_{2N}) }
{ \partial (\widetilde{\theta}_1 ,\ldots, \widetilde{\theta}_J)} \right)
\ee	
where the products are taken over all edges $e$ and tetrahedra $\Delta$ of $\calT$.
Here $e_\Delta$ is the edge of $\Delta$ that determines the two face equations
of $\widetilde{\Delta}$.
Note that entries of the Jacobian matrix in~\eqref{eqn.def2} are given by
$\BZ[t^{\pm1}]$-linear combinations of $c(e)$.
\end{definition}
For an $\SL$-representation $\rho$ of $\pi_1(M)$ we denote by
$\tau_{n,\rho}(t)$ the $\BC^n$-torsion polynomial associated to $\rho$, i.e., 
the twisted Alexander polynomial associated to
$\mathrm{Sym}^{n-1}(\rho) \otimes \alpha$. See \cite{Wada94} for the
precise definition of the twisted Alexander polynomial. Here 
$\mathrm{Sym}^{n-1}(\rho)$ is the $(n-1)$-st symmetric power of $\rho$; 
for instance, $\mathrm{Sym}^1(\rho)=\rho$ and
$\mathrm{Sym}^2(\rho) =\mathrm{Ad}\circ\rho$. 
This defines all the invariants that appear in Conjecture~\ref{conj.23}.

As a final remark in this section, if we choose different lifts for the
tetrahedra, then the sets $I$ and $J$ above will change, but they will always
satisfy $2N < |I| \leq 4N$ and $N < |J| \leq 6N$. If we wish, we may assume
that $|I| = 4N$ and $|J| = 6N$ by assigning different variables to all edges and
faces (i.e., ignoring edge or face identifications). For instance, if we have $N$
tetrahedra, there are $4N$ faces (ignoring face-pairings) and $2N$ face equations,
two from each tetrahedron. We then obtain $2N$ additional equations (of the form
$\theta_i - \theta_j t^k= 0)$ from $2N$ face-pairings, hence a total of $4N$
equations in $4N$ variables. This $4N  \times 4N$ matrix gives the
1-loop $\BC^2$-polynomial. Note that the last $2N$ rows are of the form
$(\ldots 0, 0, 0, 1, 0, \ldots, 0, -t^k, 0, 0, 0, \ldots)$, so we can apply
elementary row operations to reduce the size of this matrix to $2N$.

\subsection{Boundary obstruction classes}
\label{sub.sign}

A solution $c$ to the Ptolemy equations gives rise to an $\SL$-representation
$\rho_c$ 
of $\pi_1(M)$ whose peripheral elements $\gamma$ satisfy the condition
$\text{tr}(\rho_c(\gamma))=2$. On the other hand, any $\SL$-lift $\rho$ of the
geometric representation of a hyperbolic 3-manifold has a peripheral element
$\gamma$ with $\text{tr}(\rho(\gamma))=-2$~\cite{Calegari:real}, and hence it is
not captured by the Ptolemy equations. To allow such representations,
one needs a sign-deformed version of the Ptolemy and face equations
~\eqref{eqn.ptolemy} and~\eqref{eqn.face}. This is a subtle point, often
confused in the literature, which we now explain.

It turns out that there are two ways to insert signs to the Ptolemy
equation~\eqref{eqn.ptolemy}, one using an ``obstruction class''
$\sigma \in H^2(M, \partial M, \BZ/2 \BZ) \simeq H_1(M,\BZ/2\BZ)$ as was done 
in~\cite{GTZ15,Garoufalidis:gluing,Garoufalidis:ptolemy}, and a second one using
a ``boundary obstruction class'' $\sigma \in H^1(\partial M, \BZ/2\BZ) \newline\simeq
H_1(\partial M, \BZ/2\BZ)$ introduced in~\cite{Yoon:volume} (unfortunately, a
name for these classes was not given there). The two types of obstruction classes
are connected by the natural map
\be
\label{obsconnect}
H_1(\partial M, \BZ/2\BZ) \to H_1(M, \BZ/2\BZ)
\ee
which sends a boundary obstruction class to an obstruction class. 

A solution $c$ to the sign-deformed Ptolemy equations by a boundary obstruction
class (resp., by an obstruction class) gives rise to an $\SL$
(resp., $\PSL$)-representation $\rho_c$ of $\pi_1(M)$ that satisfies the
condition $\text{tr}(\rho_c(\gamma))=\pm 2$ for all peripheral elements $\gamma$
of $\pi_1(M)$. If $c$ is a solution to the sign-deformed Ptolemy equations by a
boundary obstruction class $\sigma$ and $\tilde c$ is the corresponding solution
under the obstruction class $\tilde\sigma$ of the map~\eqref{obsconnect}, then
the $\PSL$-representation $\rho_{\tilde c}$ is the image of the
$\SL$-representation $\rho_c$ under the canonical projection
$\SL\to\PSL$.

Coming back to the problem of sign-deforming the Ptolemy and face equations, 
we will use a cocycle representative of a boundary obstruction class $\sigma \in
H^1(\partial M, \BZ/2 \BZ)$.
Recall that the truncated triangulation of an ideal triangulation $\calT$ has short
and long edges (see Figure~\ref{fig.tetrahedron}) and that the short edges give a
triangulation of $\partial M$. Let $\sigma$ be a sign-assignment
to each short edge such that the product of three signs attached to the boundary
of every triangular face is $+1$. Then the Ptolemy equation~\eqref{eqn.ptolemy}
changes to
\be
\label{eqn.sptolemy}
P_\Delta: \quad c(e_{01}) c(e_{23})
- \frac{\sigma^2_{03}\sigma^3_{12}}{\sigma^1_{03} \sigma^0_{12}} 
c(e_{02}) c(e_{13})
+ \frac{\sigma^3_{02}\sigma^2_{13}}  { \sigma^1_{02} \sigma^0_{13}}
c(e_{03}) c(e_{12})= 0 \,.
\ee
Here $\sigma^i_{jk} \in \{\pm1\}$ is
the sign that $\sigma$ assigns at the short edge that is near to the vertex $i$
and parallel to the edge $[j,k]$; see Figure~\ref{fig.p2c} or~\ref{fig.p2c2}.
The face equations in~\eqref{eqn.face} change similarly, but now the map $\theta$ of a super-Ptolemy assignment $(c,\theta)$ is a map from the set of oriented faces of $\calT$, that is, we now distinguish  the sides of a face. 
If we denote by $f^{\pm}$ the sides of a face $f$, then
\be
\theta(f^+) = \sigma_1 \sigma_2 \sigma_3 \theta(f^-) 
\ee
where $\sigma_1 ,\sigma_2,$ and $\sigma_3$ are the signs that $\sigma$ assigns at the three short edges of the face $f$. 
In addition, the face equations in ~\eqref{eqn.face} change to
\be
\label{eqn.sface}
\begin{aligned}
E_{\Delta,f_3} & : & \frac{\sigma^1_{23}}{\sigma^0_{23}} c(e_{12}) \theta (f^+_0) 
- c(e_{02}) \theta(f^+_1) +\frac{\sigma^0_{12}}{\sigma^3_{12}} c(e_{01}) \theta(f^+_2) 
& = 0 \\
E_{\Delta,f_2} & : & c(e_{13}) \theta(f^+_0) - \frac{\sigma^3_{01}}{\sigma^2_{01}}
c(e_{03}) \theta(f^+_1) + \frac{\sigma^0_{12}}{\sigma^3_{12}} c(e_{01}) \theta(f^+_3)
& = 0 \\ 
E_{\Delta,f_1} & : &\frac{\sigma^1_{03}}{\sigma^2_{03}} c(e_{23}) \theta(f^+_0) -
\frac{\sigma^3_{01}}{\sigma^2_{01}}  c(e_{03}) \theta(f^+_2) + c(e_{02}) \theta(f^+_3)
& = 0 \\
E_{\Delta,f_0} & : &\frac{\sigma^1_{03}}{\sigma^2_{03}} c(e_{23}) \theta(f^+_1)
- c(e_{13}) \theta(f^+_2) + \frac{\sigma^1_{23}}{\sigma^0_{23}} c(e_{12}) \theta(f^+_3)
& = 0
\end{aligned}
\ee
where $f_i^+$ is the side of $f_i$ that faces front.
Note that Equations~\eqref{eqn.sptolemy} and~\eqref{eqn.sface} reduce
to~\eqref{eqn.ptolemy} and~\eqref{eqn.face}, respectively, if $\sigma$ assigns $+1$ 
to all short edges.

Replacing Equations~\eqref{eqn.ptolemy} and~\eqref{eqn.face} by their signed
versions~\eqref{eqn.sptolemy} and~\eqref{eqn.sface} above, one can repeat the
definitions of the 1-loop polynomials and their results of Sections~\ref{sub.reps}
and~\ref{sub.1loops} mutatis-mutandis (see~\cite{Yoon:volume,GY23} for details).
Since the choice of the boundary obstruction class will not play any role in the
proof of our main theorem~\ref{thm.1} given in Section~\ref{sec.layered} below,
we will drop the boundary obstruction class $\sigma$ from the notation of the Ptolemy
equations and of the 1-loop polynomials.  


\section{Layered triangulations}
\label{sec.layered}

In this section we fix a fibered 3-manifold $M=M_\varphi$ obtained by the mapping
torus of a homeomorphism $\varphi : \Sigma \rightarrow \Sigma$ of a punctured
surface. After we choose an ideal triangulation of $\Sigma$, and represent $\varphi$
as a composition of 2--2 moves on triangulations of $\Sigma$, we obtain a 
taut and layered ideal triangulation $\calT_\varphi$ of $M_\varphi$. A detailed
description of taut and layered triangulations is given in Lackenby~\cite{Lackenby},
and a computational approach is given by Bell's program \texttt{flipper}~\cite{flipper}.

\subsection{Torsion polynomials}
\label{sub.torpoly}

In this section we give determinantal formulas for the $\BC^r$-torsion polynomials
for $r=2,3$, adapted to the combinatorics of layered triangulations. 

An ideal triangulation of a punctured surface $\Sigma$ has $3n$ edges and $2n$
triangles where $-n$ is the Euler characteristic of $\Sigma$. Hence 
a super-Ptolemy assignment on (an ideal triangulation of) $\Sigma$ is given by
\be
\label{eqn.initial}
(c_1,\ldots,c_{3n}, \theta_1,\ldots,\theta_{2n}) \in (\BC^\ast)^{3n} \times \BC^{2n} \,.
\ee
Here and throughout the section, all $\theta_i$ are assigned to outward oriented faces.
Attaching a tetrahedron to a surface creates one edge and two faces. Therefore, 
we obtain new variables $c_{3n+1},\ldots,c_{3n+N}$ and $\theta_{2n+1},\ldots,
\theta_{2n+2N}$ by attaching $N$ tetrahedra on $\Sigma$. It follows from 
Equations \eqref{eqn.ptolemy} and \eqref{eqn.face} that these new variables are
inductively determined by the initial variables given in~\eqref{eqn.initial}. 
In addition, 
\be
\begin{aligned}
& c_{3n+1},\ldots,c_{3n+N} \in \BQ(c_1,\ldots, c_{3n}), \\
& \theta_{2n+1},\ldots, \theta_{2n+2N} \in\BQ(c_1,\ldots, c_{3n})
[\theta_1,\ldots,\theta_{2n}]
\end{aligned}
\ee
and each $c_i$ (resp., $\theta_j$) has degree 1 with respect to $c_1,\ldots,c_{3n}$ 
(resp., $\theta_1,\ldots,\theta_{2n}$). We denote by $(c'_{1},\ldots,c'_{3n},
\theta'_{1},\ldots,\theta'_{2n})$ the super-Ptolemy assignment on the top surface
(the surface that we obtain after attaching the tetrahedra) so that a solution to
\be
c_i=c'_{i} \quad \textrm{and} \quad \theta_j = \theta'_{j} \quad  \textrm{for} \quad
1 \leq i \leq 3n, \ \ 1 \leq j \leq 2n
\ee
gives a super-Ptolemy assignment on $\calT_\varphi$, 
and vice versa. 

We now fix a solution to $c_i = c'_{i}$ for $1 \leq i \leq 3n$, i.e., a Ptolemy
assignment $c$ on $\calT_{\varphi}$ and denote its associated representation by
$\rho_c: \pi_1(M_\varphi) \rightarrow \SL$.

\begin{lemma}
\label{lem.irr}
If $\rho_c$ is irreducible, then so is its restriction $\rho_c|_{\pi_1(\Sigma)}$ to
$\pi_1(\Sigma)$.
\end{lemma}

\begin{proof} 
Suppose the restriction $\eta=\rho_c|_{\pi_1(\Sigma)}$ is reducible. Then the induced
action of $\pi_1(\Sigma)$ on $\mathbb{C}P^1$ (a) acts trivially, (b) has a unique
fixed point, or (c) has exactly two fixed points. The first case implies that the
image of $\eta$ is in $\pm I$, hence $\rho_c$ is abelian. This contradicts to the
irreducibility of $\rho$.
For the second case, the full group $\pi_1(M)$ must fix the same fixed point,
as $\pi_1(\Sigma)$ is a normal subgroup of $\pi_1(M)$. This, again,
contradicts to the irreducibility of $\rho$.
The last case implies that  up to conjugation the image $D$ of $\eta$ consists of
diagonal matrices. On the other hand, we have
$\tr \rho_c(\gamma) = \pm 2$ and $\rho_c(\gamma) D \rho_c(\gamma)^{-1} = D$ for a
peripheral curve $\gamma$ of $M$. 
This implies $\rho_c(\gamma) =  \pm I$. As $\pi_1(M)$ is generated by $\pi_1(\Sigma)$ and one such $\gamma$, this forces $\rho_c$ to be reducible.
This completes the proof.
\end{proof}

\begin{theorem} 
\label{thm.jac3}
The $\BC^3$-torsion polynomial associated to $\rho_c$ is given by
\begin{align}
\label{eqn.jac3}
\tau_{3,\rho_c}(t)& \doteq \det \left( t I - \dfrac{\partial(c'_{1},\ldots, c'_{3n})}
{\partial(c_1,\ldots, c_{3n})}\right)  
\end{align}
provided that $\rho_c$ is irreducible.
\end{theorem}

\begin{proof}
Let $\calX=\calX(\Sigma;\SL)$ be the character variety of $\Sigma$ and $\eta \in
\calX$ be (the conjugacy class of) the restriction of $\rho_c$ to $\Sigma$, which is 
irreducible due to Lemma~\ref{lem.irr}. 
Let $H_{\eta}^\ast(\Sigma;\sl)$ be the twisted cohomology of $\Sigma$ whose
coefficient $\sl$ is twisted by $\mathrm{Ad} \circ \eta$. Since $\Sigma$ is a 
punctured surface, $H^i_\eta(\Sigma;\sl)$ is trivial for all $i\geq2$. Also, since 
$\eta$ is irreducible, we have
\be
H^0_\eta(\Sigma;\sl) = 0, \quad H^1_{\eta}(\Sigma; \sl) \simeq T_\eta \calX
\ee
where the latter is Weil's isomorphism~\cite{Weil64}  (see also \cite{Sikora2012}).
From the usual Mayer-Vietoris sequence for $S^1$-fibrations, we obtain
(see e.g.~\cite{Fried86})
\be
\label{eqn.fried}
\tau_{3,\rho}(t) = \det (t I - \varphi_\calX)
\ee
where $\varphi_\calX : T_\eta \calX \rightarrow T_\eta \calX$ is the isomorphism
induced by the monodromy $\varphi : \Sigma \rightarrow \Sigma$. Note that the
action on $\calX$ induced by $\varphi$ fixes $\eta$ and its derivative at $\eta$ 
is $\varphi_\calX$.
 
We first consider the case when $\varphi$ fixes each puncture of $\Sigma$. 
Let $\gamma_i$ for $1 \leq i \leq p$ be a small loop that winds around each puncture of $\Sigma$. Here $p$ is the number of punctures of $\Sigma$.
It is clear from Figure~\ref{fig.p2c} that the image of the map $F : (\BC^\ast)^{3n} \rightarrow \calX$, sending a Ptolemy assignment on $\Sigma$ to its associated representation (see Remark~\ref{rmk.surface}) is contained in
\be
\label{eqn.trace}
\calY = \left \{ \eta' \in \calX \, | \, \tr \eta'(\gamma_1 ) = \cdots
  = \tr \eta'(\gamma_p)= 2 \right \} \subset \calX   \,.
\ee
Note that $\eta \in \calY$ and that the image of $F$ is Zariski-open in $\calY$ 
(see \cite{GTZ15} or \cite{Garoufalidis:ptolemy}). The action on $\calX$ induced by
$\varphi$ preserves the trace function of each $\gamma_i$, hence $\calY$. This
implies that
\be
\label{eqn.proof1}
\det (t I -  \varphi_\calX)=(t-1)^p\det \left(t I - \varphi_\calY\right) 
\ee
where  $\varphi_\calY$ is the restriction of $\varphi_\calX$ to ${T_{\eta}\calY}$.
Note that $\dim T_\eta \calY= \dim T_\eta \calX - p = 3n-p$.

Recall that we have two coordinates $(c_1,\ldots,c_{3n})$ and
$(c'_{1},\ldots,c'_{3n})$ for $(\BC^\ast)^{3n}$. Since we fixed a solution, say $c$,
to $c_i = c'_{i}$ for $1 \leq i \leq 3n$ with $F(c) = \eta$, we have a
commutative diagram for tangent spaces:
\be
\label{eqn.diagram}
\begin{tikzcd}[column sep=2cm]
T_{c}(\BC^\ast)^{3n}\arrow[d,"dF"] \arrow[r, "\partial c'_{i} / \partial c_j"] &
T_{c}(\BC^\ast)^{3n}  \arrow[d, "dF"] \\
T_{\eta}\calY \arrow[r,"\varphi_{\calY}"] & T_{\eta}\calY
\end{tikzcd}
\ee

We now fix one puncture of $\Sigma$ and let $d_i \in \{0,1,2\}$ be the number of
vertices of the edge $c_i$ that are the chosen puncture. The action~\eqref{eqn.act}
on $(\BC^\ast)^{3n}$ says that
\be
c'_{i}(k^{d_1} c_1,\ldots,k^{d_{3n}}c_{3n}) = k^{d_i} c'_{i}(c_1,\ldots,c_{3n}),
\quad \quad k \in \BC^\ast,  \quad  1 \leq i \leq 3n \, .
\ee
Here we view $c'_i$ as a function of $c_1,\ldots, c_{3n}$. It follows that the Jacobian
$\partial c'_{i}/\partial c_j$ has an eigenvalue $1$ with an eigenvector 
$(d_1 c_1,\ldots, d_{3n} c_{3n})$. Since $\Sigma$ has $p$ punctures, we obtain
$p$ eigenvectors $V_1,\ldots, V_p$ with eigenvalues $1$. Furthermore, one checks 
that these eigenvectors are linearly independent from the fact that $c_1,\ldots,c_{3n}$
are edges of triangles. On the other hand, the action~\eqref{eqn.act} preserves
$\SL$-representations, i.e.,
\be
F(c_1,\ldots,c_{3n})=F(k^{d_1}c_1,\ldots,k^{d_{3n}}c_{3n}), \quad k \in \BC^\ast,
\ee
hence the vertical map $dF$ in~\eqref{eqn.diagram} induces a linear map
$T_c((\BC^\ast)^{3n}/(\BC^\ast)^{p}) \rightarrow T_\eta \calY$
between $(3n-p)$-dimensional vector spaces. Since the image of $F$ is Zariski-open
in $\calY$, the induced map is an isomorphism, and thus
\be
\label{eqn.proof2}
\det \left( t I- \dfrac{\partial(c'_{1},\ldots, c'_{3n})}{\partial(c_1,\ldots, c_{3n})}
\right)
=(t-1)^p\det \left(t I -  \varphi_{\calY}\right)\,.
\ee
Then the desired formula~\eqref{eqn.jac3} follows from~\eqref{eqn.fried},
\eqref{eqn.proof1} and~\eqref{eqn.proof2}.

We now consider the case when $\varphi$ permutes some punctures of $\Sigma$. 
As the action on $\calX$ induced by $\varphi$ permutes the trace functions of $\gamma_1,\ldots,\gamma_p$, we obtain a $p\times p$-permutation matrix $P$
with 
\be
\label{eqn.proof3}
\det (t I -  \varphi_\calX)=\det(tI-P)\det \left(t I - \varphi_\calY\right) 
\ee
instead of Equation~\eqref{eqn.proof1}.
Similarly, the vectors $V_1,\ldots, V_p$ would be no longer eigenvectors, 
but the Jacobian $\partial c'_i/ \partial c_j$ permutes them with the same
permutation matrix $P$. Namely,
\be
\label{eqn.proof4}
\det \left( t I- \dfrac{\partial(c'_{1},\ldots, c'_{3n})}{\partial(c_1,\ldots, c_{3n})}
\right)
=\det(t I - P)\det \left(t I -  \varphi_{\calY}\right)\,.
\ee
Combining the above two equations, we obtain the desired formula~\eqref{eqn.jac2}.
\end{proof}

\begin{theorem} 
\label{thm.jac2}
The $\BC^2$-torsion polynomial associated to $\rho_c$ is given by
\begin{align}
\label{eqn.jac2}
\tau_{2,\rho_c}(t)& \doteq \det \left( t I - \dfrac{\partial(\theta'_{1},\ldots, \theta'_{2n})}
{\partial(\theta_1,\ldots, \theta_{2n})}\right) 
\end{align}
provided that $\rho_c$ is irreducible.
\end{theorem}
Note that $\theta'_{i}$
is linear in $\theta_1,\ldots, \theta_{2n}$, hence
$\partial \theta'_{i}/\partial \theta_j$ is in $\BQ(c_1,\ldots,c_{3n})$; in particular,
the determinant in~\eqref{eqn.jac2} only depends on the Ptolemy assignment.

\begin{proof}
Let $\calX=\calX(\Sigma;\SL)$ be the character variety of $\Sigma$ and $\eta \in
\calX$ be (the conjugacy class of) the restriction of $\rho_c$ to $\Sigma$, which is 
irreducible due to Lemma~\ref{lem.irr}. Let
$H_{\eta}^\ast(\Sigma;\BC^2)$ be the twisted cohomology of $\Sigma$ whose
coefficient $\BC^2$ is twisted by $\eta$. Since $\Sigma$ is a punctured surface
and $\eta$ is irreducible, $H^i_\eta(\Sigma;\BC^2)$ is trivial for all $i$ but $i=1$.
We first claim that
\be
H_{\eta}^1(\Sigma;\BC^2) \simeq  T_\eta \calZ \quad \textrm{ for } \calZ := p^{-1}(\eta)
\ee
where $p :  \calX(\Sigma;\SLC)  \rightarrow \calX$ is the natural
projection map. Here we identify $\eta$ with $(\eta,0) : \pi_1(\Sigma) \rightarrow \SLC$.

An element of $\calZ$ is represented by a pair $(\eta, v)$ where $v : \pi_1(\Sigma)
\rightarrow \BC^2$ is a 1-cocycle, i.e., a map satisfying
\be 
\label{eqn.cocycle}
v(a b) = v(a) + \eta(a) v(b) \quad \textrm{ for all } a,b \in \pi_1(\Sigma) \,.
\ee
As the condition~\eqref{eqn.cocycle} is linear in $v$, a vector in $T_\eta \calZ$ is
also represented by a 1-cocycle. We define a map from $T_\eta \calZ$ to
$H^1_\eta(\Sigma;\BC^2)$ naturally by sending a 1-cocycle to its class in the first
cohomology. Then a routine computation shows that this map is an isomorphism.
Note that $\dim T_\eta \calZ = \dim H^1(\Sigma;\BC^2) = 2n$.

From the usual
Mayer-Vietoris sequence for $S^1$-fibrations, we obtain (see e.g.~\cite{Fried86}) 
\be
\label{eqn.Fried}
\tau_{2,\rho}(t) =  \det (t I - \varphi_\calZ)
\ee
where $\varphi_\calZ : T_\eta \calZ \rightarrow T_\eta \calZ$ is the isomorphism
induced by the monodromy $\varphi$. Note that  $\varphi$ acts on $\calX$ fixing
$\eta$, hence it induces an action on $\calZ$, whose derivative at $\eta$ is
$\varphi_\calZ$.

Recall that we have a fixed Ptolemy assignment, say  $c$, on $\Sigma$ whose associated
representation is $\eta$. Let  $G : \BC^{2n} \rightarrow \calZ$ be a map sending
$\theta \in \BC^{2n}$ to the representation associated to the super-Ptolemy assignment
${(c,\theta)}$ (see Remark~\ref{rmk.surface}). It is clear that the image of $G$ is contained in $\calZ$ with
$G(0)=\eta$. In addition, it is proved in \cite{GY23} that the image of $G$ is 
Zariski-open in $\calZ$. On the other hand, we obtain a commutative diagram for
tangent spaces from  two coordinates $(\theta_1,\ldots,\theta_{2n})$  and
$(\theta'_{1},\ldots,\theta'_{2n})$ for $\BC^{2n}$:
\be
\label{eqn.diagram2}
\begin{tikzcd}[column sep=2cm]
T_{0}\BC^{2n}\arrow[d,"dG"] \arrow[r, "\partial \theta'_{i} / \partial \theta_j"] &
T_{0}\BC^{2n}\arrow[d, "dG"] \\
T_{\eta}\calZ \arrow[r,"\varphi_\calZ"] & T_{\eta}\calZ
\end{tikzcd}
\ee
Then the desired formula~\eqref{eqn.jac2} follows from \eqref{eqn.Fried} and
\eqref{eqn.diagram2}. 
\end{proof}

\subsection{Proof of Theorem~\ref{thm.1}}
\label{sub.thm1}

In this section we compare the formulas for the torsion polynomials given in
Theorems~\ref{thm.jac3} and~\ref{thm.jac2} with analogous formulas for the
1-loop polynomials to deduce Theorem~\ref{thm.1}. 

\begin{claim}
\label{claim1}
Equation~\eqref{conj23} holds for $n=3$ and for all Ptolemy assignments $c$ on
$\calT_\varphi$ with $\rho_c$  irreducible.
\end{claim}

\begin{proof}

We label the edges of $\calT_\varphi$ as in Section~\ref{sub.torpoly}: the edges of
$\Sigma$ are labeled with $c_1,\ldots, c_{3n}$ and the new edge created when we
attach the $i$-th tetrahedron $\Delta_i$ is labeled by $c_{3n+i}$ for $1 \leq i \leq N$.
Note that some edges have multiple labels.
The Ptolemy equation $P_i$ of $\Delta_i$ is written as
\be
P_i : \quad c_{T(i)} c_{B(i)} - c_{E_1(i)}c_{E_2(i)} + c_{E_3(i)} c_{E_4(i)} = 0
\ee
where $c_{T(i)}$ and $c_{B(i)}$ are the top and bottom edges  and 
$c_{E_1(i)} ,\ldots, c_{E_4(i)}$ are the equatorial edges of $\Delta_i$. 
Note that $T(i)=3n+i$ and $B(i), E_1(i),\ldots, E_4(i) \leq 3n+i-1$ and that
$\{c_{B(1)},\ldots,c_{B(N)}\}$  is the edge set of $\calT_\varphi$. 
Then, as we explained in Section~\ref{sub.1loops},  the 1-loop polynomial $\delta_{3,c}(t)$
is given by
\begin{align}
\delta_{3,c}(t) & \doteq  \left(\prod_{i=1}^N \frac{1}{c_{B(i)}}\right) 
\det \left( \dfrac{ \partial ( Q_1,\ldots,Q_{3n},P_1,\ldots, P_N)}{
\partial ( c_1,\ldots, c_{3n+N})} \right)
\end{align}
where $Q_i = t c_i - c'_{i}$ for $1 \leq i \leq 3n$ where $c'_{i}$ is the
Ptolemy variable on the top surface at the same position as $c_i$.
Letting $P'_i = P_i/c_{B(i)}$ for $ 1 \leq i \leq N$, we have
\be
\label{eqn.matrix}
\delta_{3,c}(t) \doteq \det \left(
  \dfrac{ \partial ( Q_1,\ldots, Q_{3n},P'_1,\ldots, P'_N)} {\partial
    ( c_1,\ldots, c_{3n+N})} \right) =\det
\begin{pmatrix} 
t &  & \ast &\rvline &  & &  \\
& \ddots &  &\rvline &  & \ast &  \\
\ast &  &  t &\rvline &  & &  \\
\hline
&  &  &\rvline & 1 & & 0 \\
& \ast & &\rvline &  &  \ddots  &\\
& &  &\rvline & \ast & & 1 
\end{pmatrix} \, .
\ee
Using a determinant formula for block matrices
\be
\det
\begin{pmatrix}
A & \rvline & b \\
\hline
c & \rvline & 1
\end{pmatrix}
= \det (A - b c) ,
\ee
the determinant in~\eqref{eqn.matrix} equals to
\begin{align}
\det \left( \dfrac{ \partial ( Q_1,\ldots, Q_{3n},P'_1,\ldots, P'_{N-1})}{
\partial ( c_1,\ldots, c_{3n+N-1})} -  
\begin{pmatrix}
\frac{\partial Q_1}{\partial c_{3n+N}}  \\
\vdots \\
 \frac{\partial Q_{3n}}{\partial c_{3n+N}} \\
0 \\
 \vdots \\
0
\end{pmatrix}
\begin{pmatrix}
 \frac{\partial P'_N}{\partial c_1} &
 \cdots &
  \frac{\partial P'_N}{\partial c_{3n+N-1}} 
\end{pmatrix}
\right) \,. \label{eqn.det1}
\end{align}
On the other hand, solving the equation $P'_N=0$ gives
\be
\label{eqn.sol}
c_{3n+N} = \frac{c_{E_1(N)}c_{E_2(N)} - c_{E_3(N)} c_{E_4(N)}}{c_{B(N)}} \, ,
\ee 
hence  $\partial P'_N/\partial c_j = - \partial c_{3n+N}/\partial c_j$ for
all $j \leq 3n+N-1$. It follows that the determinant~\eqref{eqn.det1} equals to
\begin{align}
\det \left( \dfrac{ \partial ( Q_1,\ldots Q_{3n},P'_1,\ldots P'_{N-1})}{
\partial ( c_1,\ldots, c_{3n+N-1})} +  
\begin{pmatrix}
\frac{\partial Q_1}{\partial c_{3n+N}}  \\
\vdots \\
\frac{\partial Q_{3n}}{\partial c_{3n+N}} \\
0 \\
\vdots \\
0
\end{pmatrix}
\begin{pmatrix}
\frac{\partial c_{3n+N}}{\partial c_1} &
\cdots &
\frac{\partial c_{3n+N}}{\partial c_{3n+N-1}} 
\end{pmatrix}
\right) \,. \label{eqn.det2}
\end{align}
On the other hand, the chain rule says that  for any function $f$ we have
\be
\frac{\partial f }{\partial c_i} +\frac{\partial f }{\partial c_{3n+N}}
\frac{\partial c_{3n+N}}{\partial c_i }
= \frac{\partial \overline{f} }{\partial c_i}
\ee
where $\overline{f}$ is obtained from $f$ by eliminating $c_{3n+N}$ by
using~\eqref{eqn.sol}. It follows that  the determinant~\eqref{eqn.det2} is
simplified to
\be
\det \left( \dfrac{ \partial ( \overline{Q}_1,\ldots,
    \overline{Q}_{3n},P'_1,\ldots , P'_{N-1})} {\partial ( c_1,\ldots, c_{3n+N-1})}
\right)
\ee
where $\overline{Q}_i$ are obtained from $Q_i$ by eliminating $c_{3n+N}$ by
using~\eqref{eqn.sol}. For simplicity we let $Q_i = \overline{Q}_i$, regarding them now
 as functions in $c_1,\ldots, c_{3n+N-1}$.
Since $P'_1,\ldots,P'_{N-1}$ are unchanged, we can apply this reduction
until we remove all $P'_i$ and $c_{3n+i}$ for $i \geq 1$.  It follows that 
\be
\label{eqn.Qt}
\begin{aligned}
\delta_{3,c}(t) \doteq \det \left( \dfrac{ \partial (Q_1,\ldots Q_{3n})}{
\partial ( c_1,\ldots, c_{3n})} \right)  
\end{aligned}
\ee
where $Q_i(t) = t c_i - c'_i$ are functions in $c_1,\ldots,c_{3n}$.
This completes the proof, since Equation~\eqref{eqn.Qt} equals to the
$\BC^3$-torsion polynomial given in Theorem~\ref{thm.jac3}.
\end{proof}

\begin{claim}
Equation~\eqref{conj23} holds for $n=2$ and for all Ptolemy assignments $c$ on
$\calT_\varphi$ with $\rho_c$ irreducible.
\end{claim}

\begin{proof}
We label the faces of $\calT_\varphi$ as in Section~\ref{sub.torpoly}: the faces of
$\Sigma$ are labeled with $\theta_1,\ldots, \theta_{2n}$ and the new two faces
created when we attach the $i$-th tetrahedron $\Delta_i$ is labeled by
$\theta_{2n+2i-1}$ and $\theta_{2n+2i}$ for $1 \leq i \leq N$.
Note that some faces have multiple labels.
Then the face equations $E_{2i-1}$ and $E_{2i}$ for the top faces of $\Delta_i$ are
written as
\be
\begin{aligned}
E_{2i-1} &: &  c_{T(i)} \theta_{2n+2i-1} - c_{E_1(i)} \theta_{\alpha(i)} + c_{E_2(i)}
\theta_{\beta(i)} &=& 0 \\
E_{2i} &: & c_{T(i)}  \theta_{2n+2i} - c_{E_3(i)}\theta_{\alpha(i)} + c_{E_4(i)}
\theta_{\beta(i)}  &=& 0
\end{aligned}
\ee
where $\theta_{\alpha(i)}$ and $\theta_{\beta(i)}$ are the bottom faces of $\Delta_i$
and $c_{E_1(i)} ,\ldots, c_{E_4(i)}$ are the equatorial edges of $\Delta_i$. 
Note that $\alpha(i), \beta(i) \leq 2n+2i-2$ and that $\{c_{T(1)},\ldots, c_{T(N)}\}$ is
the edge set of $\calT_\varphi$.
Then, as we explained in Section~\ref{sub.1loops},  the 1-loop polynomial
$\delta_{2,c}(t)$ is given by
\begin{align}
\delta_{2,c}(t) &\doteq   
\left(\prod_{i=1}^N \frac{1}{c_{T(i)}}\right)^2
\det \left(
\dfrac{ \partial ( F_1,\ldots F_{2n},E_1,\ldots E_{2N})}{
\partial ( \theta_1,\ldots, \theta_{2n+2N})} \right)
\end{align}
where $F_i = t \theta_i - \theta'_{i}$ for $1 \leq i \leq 2n$ where $\theta'_{i}$
is the super-Ptolemy variable on the top surface at the same position as $\theta_i$.
It follows that 
\be
\label{eqn.matrix2}
\delta_{2,c}(t) \doteq \det \left( \dfrac{ \partial (
F_1,\ldots F_{2n},E'_1,\ldots E'_{2N})} {\partial (
\theta_1,\ldots, \theta_{2n+2N})} \right) =\det
\begin{pmatrix} 
t &  & \ast &\rvline &  & &  \\
& \ddots &  &\rvline &  & \ast &  \\
\ast &  &  t &\rvline &  & &  \\
\hline
&  &  &\rvline & 1 & & 0 \\
& \ast & &\rvline &  &  \ddots  &\\
& &  &\rvline & \ast & & 1 
\end{pmatrix} \, .
\ee
where $E'_{2i-1} = E_{2i-1}/c_{T(i)}$ and $E'_{2i} = E_{2i}/c_{T(i)}$ for
$1 \leq i \leq N$. Then the same reduction that we used in the proof of
Claim~\ref{claim1} shows that 
\be
\label{eqn.Et}
\begin{aligned}
\delta_{2,c}(t) & \doteq \det \left( \dfrac{
\partial (F_1,\ldots F_{2n})}{\partial ( \theta_1,\ldots, \theta_{2n})} \right)  
\end{aligned}
\ee
where $F_i = t \theta_i - \theta'_i$ are now functions in $\theta_1,\ldots,\theta_{2n}$.
This completes the proof, since Equation~\eqref{eqn.Et} equals to the $\BC^2$-torsion
polynomial given in Theorem~\ref{thm.jac2}.
\end{proof}

We now discuss the second comment after the statement of Theorem~\ref{thm.1}.
When $M_\varphi$ is hyperbolic, the geometric $\PSL$-representation lifts
to an $\SL$-representation (see~\cite{Culler:lifting}) and every such lift 
$\rho^\geom$ is of the form $\rho_{c^\geom}$ for a Ptolemy assignment $c^\geom$
on $\calT_\varphi$. The latter claim follows from the fact that the edges of
$\calT_\varphi$ are ideal arcs in $\Sigma$, hence homotopically
non-peripheral~\cite{Dunfield:incompressibility,Hodgson:essential}. Since every
lift $\rho^\geom$ is clearly irreducible, we obtain 
\be
\label{eqn.geompoly}
\delta_{n,c^\geom}(t) \doteq \tau_{n,\rho^\geom}(t) 
\ee
for $n=2,3$, as a special case of Theorem~\ref{thm.1}.


\section{Examples}
\label{sec.examples}

A census of oriented cusped hyperbolic manifolds \texttt{OrientableCuspedCensus}
is given by \texttt{SnapPy}~\cite{snappy}, and of those, a list \texttt{CHW} of
fibered ones is given by \texttt{flipper}~\cite{flipper}. In this section we
illustrate our theorems~\ref{thm.jac3} and~\ref{thm.jac2} with one of the first
few manifolds from this list, namely for the fibered 3-manifold
$M=\texttt{m036}$, a one-cusped hyperbolic 3-manifold with 4 tetrahedra and
homology $H_1(M,\BZ)=\BZ+\BZ/3\BZ$, hence with two obstruction classes, a
trivial one and a non-trivial one, and four boundary obstruction classes. 

\subsection{The layered triangulation of \texttt{m036}}
\label{sub.m036layered}

The manifold $M=M_\varphi$ is fibered where its fiber $\Sigma=\Sigma_{2,1}$
is a once-punctured
surface of genus 2 and its pseudo-Anosov monodromy is $\varphi=\texttt{aaabcd}$,
where $\texttt{a},\texttt{b},\texttt{c},\texttt{d}$ are the positive Dehn twists
on $\Sigma$ shown in Figure~\ref{fig.Sigma21}. 

\begin{figure}[htpb!]
\begingroup%
  \makeatletter%
  \providecommand\color[2][]{%
    \errmessage{(Inkscape) Color is used for the text in Inkscape, but the package 'color.sty' is not loaded}%
    \renewcommand\color[2][]{}%
  }%
  \providecommand\transparent[1]{%
    \errmessage{(Inkscape) Transparency is used (non-zero) for the text in Inkscape, but the package 'transparent.sty' is not loaded}%
    \renewcommand\transparent[1]{}%
  }%
  \providecommand\rotatebox[2]{#2}%
  \newcommand*\fsize{\dimexpr\f@size pt\relax}%
  \newcommand*\lineheight[1]{\fontsize{\fsize}{#1\fsize}\selectfont}%
  \ifx\svgwidth\undefined%
    \setlength{\unitlength}{164.39845276bp}%
    \ifx\svgscale\undefined%
      \relax%
    \else%
      \setlength{\unitlength}{\unitlength * \real{\svgscale}}%
    \fi%
  \else%
    \setlength{\unitlength}{\svgwidth}%
  \fi%
  \global\let\svgwidth\undefined%
  \global\let\svgscale\undefined%
  \makeatother%
  \begin{picture}(1,0.6040536)%
    \lineheight{1}%
    \setlength\tabcolsep{0pt}%
    \put(0,0){\includegraphics[width=\unitlength,page=1]{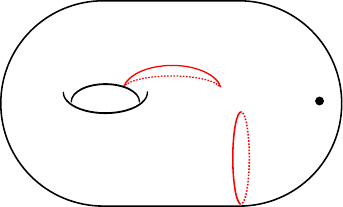}}%
    \put(0.13010111,0.39173186){\color[rgb]{0,0,0}\makebox(0,0)[lt]{\lineheight{0}\smash{\begin{tabular}[t]{l}$\texttt{a}$\end{tabular}}}}%
    \put(0.48672772,0.43404846){\color[rgb]{0,0,0}\makebox(0,0)[lt]{\lineheight{0}\smash{\begin{tabular}[t]{l}$\texttt{b}$\end{tabular}}}}%
    \put(0,0){\includegraphics[width=\unitlength,page=2]{S_2_1.pdf}}%
    \put(0.66746957,0.50524309){\color[rgb]{0,0,0}\makebox(0,0)[lt]{\lineheight{0}\smash{\begin{tabular}[t]{l}$\texttt{d}$\end{tabular}}}}%
    \put(0.26393926,0.07429551){\color[rgb]{0,0,0}\makebox(0,0)[lt]{\lineheight{0}\smash{\begin{tabular}[t]{l}$\texttt{f}$\end{tabular}}}}%
    \put(0.61424199,0.06777833){\color[rgb]{0,0,0}\makebox(0,0)[lt]{\lineheight{0}\smash{\begin{tabular}[t]{l}$\texttt{e}$\end{tabular}}}}%
    \put(0.82485488,0.20594425){\color[rgb]{0,0,0}\makebox(0,0)[lt]{\lineheight{0}\smash{\begin{tabular}[t]{l}$\texttt{c}$\end{tabular}}}}%
    \put(0,0){\includegraphics[width=\unitlength,page=3]{S_2_1.pdf}}%
  \end{picture}%
\endgroup%

\caption{Dehn twists on $\Sigma_{2,1}$.}
\label{fig.Sigma21}
\end{figure}

Among other things, \texttt{flipper} describes a triangulation of $\Sigma$,
the monodromy $\varphi$ and the layered triangulation. To begin with, the
triangulated surface $\Sigma$ is described by its list of 6 triangles
\verb|(~8, ~1, ~4)| $,\ldots,$ \verb|(4, 5, 6)|, each encoded by a triple of
edges taken counterclockwise. Note that $\Sigma$ has 9 oriented edges \texttt{0}
$,\ldots,$ \texttt{8}
(where \verb|~e| denotes the orientation-reversed edge $e$).

\begin{lstlisting}
sage: import snappy, flipper
sage: M = snappy.Manifold('m036')
sage:  phi = flipper.monodromy_from_bundle(M)
sage: S = phi.source_triangulation; S
[(~8, ~1, ~4),(~7, ~3, 2),(~6, ~2, 1),(~5, 0, 3),(~0, 7, 8),(4, 5, 6)]
\end{lstlisting}

The monodromy $\varphi$ is given by a product of moves that consist of
combinatorial surface automorphisms and flips

\begin{lstlisting}
  sage: phi.sequence
[Isometry [1, 2, 3, 4, 5, 6, 7, 8, ~0], Flip 8, Flip 5, Flip 7, Flip 4]
\end{lstlisting}

Each flip adds a tetrahedron to the layered triangulation. Explicitly, the
sequence of triangulations of $\Sigma$ under these moves is given by

\begin{lstlisting}
sage: for flip in phi.sequence[1:]:
....:     print(flip.target_triangulation)
....: 
[(~8, 2, ~4),(~7, ~0, ~3),(~6, ~2, 1),(~5, ~1, 0),(3, 4, 5),(6, 7, 8)]
[(~8, ~4, 6),(~7, ~0, ~3),(~6, ~2, 1),(~5, ~1, 0),(2, 8, 7),(3, 4, 5)]
[(~8, ~4, 6),(~7, ~0, ~3),(~6, ~2, 1),(~5, 0, 3),(~1, 5, 4),(2, 8, 7)]
[(~8, ~4, 6),(~7, ~3, 2),(~6, ~2, 1),(~5, 0, 3),(~1, 5, 4),(~0, 7, 8)]
\end{lstlisting}

Starting with the triangulation of $\Sigma$ after the combinatorial surface
automorphism, the sequence of 4 flips is shown in Figure~\ref{fig.m036}.
Note that \texttt{flipper} reuses labels for edges; for instance, the first flip make the
edge \texttt{8} disappear and create a new edge, which is also labeled with \texttt{8}.
However, to avoid confusion we will give new labels for new edges.

\begin{figure}[htpb!]
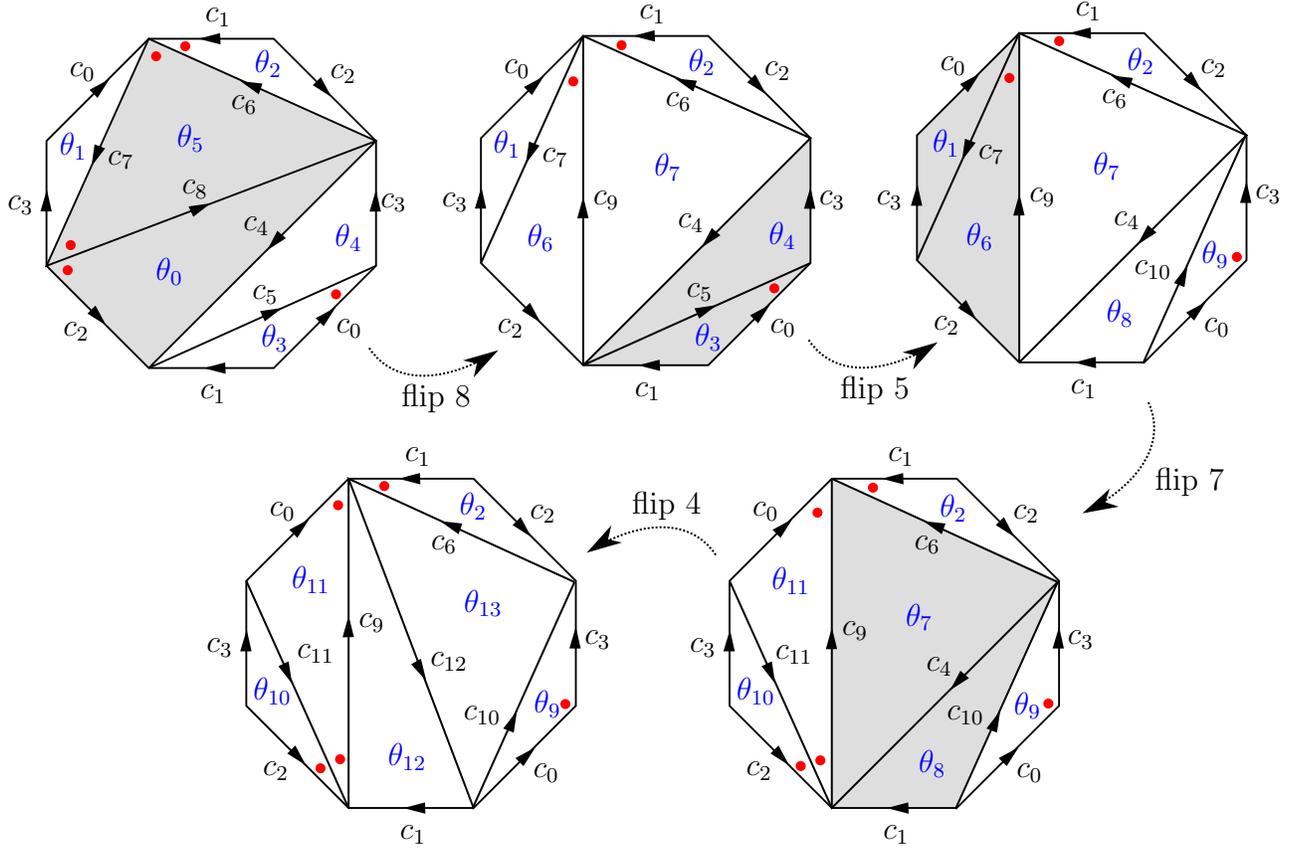
\caption{Four flips for $\texttt{m036}$.}
\label{fig.m036}
\end{figure}

\subsection{$\BC^3$-torsion polynomial}
\label{sub.m036_3}

In this section we compute the $\BC^3$-torsion polynomial.
As shown in Figure~\ref{fig.m036}, the initial surface has nine Ptolemy variables
$c_0,\ldots,c_8$.  Attaching four tetrahedra to the surface, we obtain four
additional variables $c_9, \ldots, c_{12}$ with four Ptolemy equations:
\be
\label{eqn.ptolemy036}
\begin{aligned}
P_{1} : \quad \quad& c_9 c_8 - c_2 c_6 - c_4 c_7 &=0\\
P_{2} :\quad \quad& c_{10} c_{5} + c_1 c_3 + c_0 c_4 &=0\\
P_{3} :\quad \quad & c_{11} c_{7} + c_2 c_0 - c_3 c_9&=0\\
P_{4}  :\quad \quad& c_{12} c_4 + c_1 c_6 - c_9 c_{10}&=0
\end{aligned} 
 \ee
Solving these equations determines  the additional variables $c_9,\ldots,c_{12}$ as
\begin{small}
\be
\begin{aligned}
( c_9,\ldots,c_{12}) &= \left( \frac{c_2 c_6 + c_4 c_7}{c_8},  \
  -\frac{c_1 c_3 + c_0 c_4}{c_5}, \
  \frac{c_2 c_3 c_6 + c_3 c_4 c_7 - c_0 c_2 c_8}{c_7 c_8}, \right. \\
& \qquad \qquad \qquad \left. - \frac{ c_0 c_4 (c_2 c_6 + c_4 c_7)
    + c_1 (c_2 c_3 c_6 + c_3 c_4 c_7 + c_5 c_6 c_8)}{ c_4 c_5 c_8}   \right) \, .
 \end{aligned}
\ee
\end{small}

We isotope the initial surface as in Figure~\ref{fig.m036id} so that it can be
identified with the terminal surface given in Figure~\ref{fig.m036} in an obvious
way. As in Section~\ref{sec.layered}, we denote by $(c'_0,\ldots,c'_8)$ the
Ptolemy variables on the terminal surface
\be
\label{eqn.id}
(c'_0, c'_1, c'_2, c'_3, c'_4, c'_5, c'_6, c'_7, c'_8)=
(c_1, c_2, c_3, c_{12}, c_{10}, c_6, c_{11}, c_9, -c_0)\,.
\ee
so that solving $c'_i = c_i$ for $0 \leq i\leq 8$ gives a Ptolemy assignment
on $\calT_\varphi$. Precisely, the solutions to $c'_i = c_i$ for $0 \leq i\leq 8$
are
\begin{equation}
\label{eqn.csol}
(c_0, c_1, c_2, c_3, c_4, c_5, c_6, c_7, c_8) = c \,
(1, \, 1, \, 1, \, 1, \, \beta, \, 1-\beta, \, 1-\beta, \, \beta-2, \, -1),
\quad c\in\BC^\ast
\end{equation}
where $\beta$ is a solution to
\be
\label{m036beta}
\beta^2 -2\beta -1 =0 \,.
\ee
Applying Theorem~\ref{thm.jac3}, we obtain 
\begin{align*}
\tau_{3,\rho_c}(t)&=\det \left( t I - \dfrac{\partial c'_{i}}
{\partial c_j}\right)=-1 +4t+2t^3 +t^4 +t^5 +2t^6+4t^8+t^9\,.
\end{align*}
Note that this polynomial does not come from a lift of the geometric representation,
as such a lift always has a peripheral curve whose image under the lift has trace
$-2$.

\begin{figure}[htpb!]
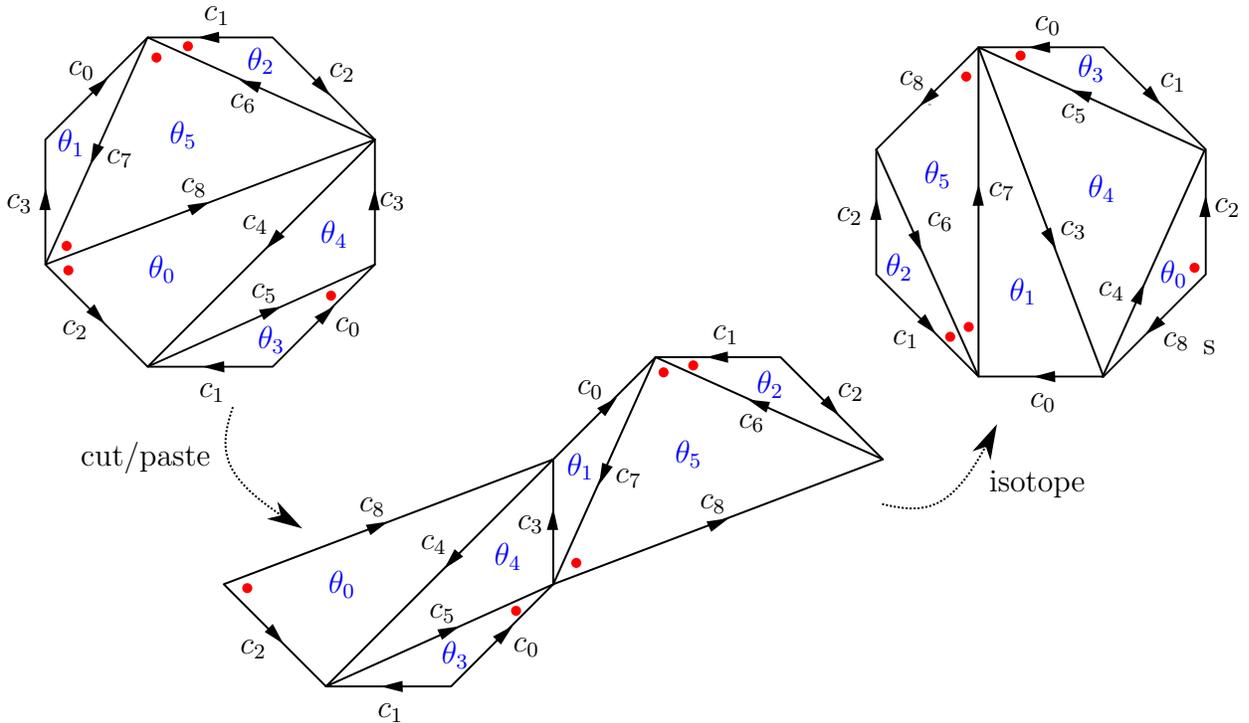
\caption{Modifying the initial surface.}
\label{fig.m036id}
\end{figure}

To cover a lift of the geometric representation, we need the sign-deformation
of the equations as in 
Section~\ref{sub.sign}. Recall that this is done by assigning a sign to
every short edge of (the truncated triangulation of) $\calT_{\varphi}$. As shown in
Figure~\ref{fig.m036}, we assign $-1$ if a short edge is red-dotted, and $+1$,
otherwise. Note that each surface in Figure~\ref{fig.m036} has an odd number of
red dots, and this implies that a curve that winds the puncture of each surface maps
to an $\SL$-matrix of trace $-2$ (we refer to \cite{Yoon21} for details).
According to this sign-assignment, the Ptolemy equations~\eqref{eqn.ptolemy036}
change to
\be
\label{eqn.ptolemy036sign}
\begin{aligned}
P_{1} : \quad \quad & c_9 c_8 + c_2 c_6 - c_4 c_7      &=0 \\
P_{2} : \quad \quad & c_{10} c_{5} + c_1 c_3 - c_0 c_4  &=0 \\
P_{3} : \quad \quad & c_{11} c_{7} - c_2 c_0 - c_3 c_9  &=0 \\
P_{4} : \quad \quad & c_{12} c_4 + c_1 c_6 - c_9 c_{10} &=0 
\end{aligned} 
\ee
with
\begin{small}
\be
\begin{aligned}
( c_9,\ldots,c_{12}) &=& \left( \frac{-c_2 c_6 + c_4 c_7}{c_8},  \
  \frac{-c_1 c_3 + c_0 c_4}{c_5}, \
  \frac{-c_2 c_3 c_6 + c_3 c_4 c_7 + c_0 c_2 c_8}{c_7 c_8}, \right. \\
& &  \left. \frac{ c_0 c_4 (-c_2 c_6 + c_4 c_7)
+ c_1 (c_2 c_3 c_6 - c_3 c_4 c_7 - c_5 c_6 c_8)}{ c_4 c_5 c_8} \right) \,.
\end{aligned}
\ee
\end{small}
Then solving $c'_i = c_i$ for $0 \leq i\leq 8$ we obtain
\begin{equation}
\label{eqn.csolsign}
(c_0, c_1, c_2, c_3, c_4, c_5, c_6, c_7, c_8) = c \,
(1, \, 1, \, 1, \, 1, \, \a, \, -\a-\a^2, \, -\a-\a^2 ,\, -\a, \, -1),
\quad c \in \BC^\ast
\end{equation}
where $\a$ is a solution to
\be
\label{m036alpha}
\a^3+\a^2+\a-1=0 \,, \qquad \alpha^\geom \approx-0.77184 + 1.11514 \cdot i
\ee
with $\a^\geom$ corresponding to (a lift of) the geometric representation. Applying
Theorem~\ref{thm.jac3}, we obtain 
\begin{small}
\begin{align*}
\tau_{3,\rho_c}(t)&=
 -1 -(2\a^2 + 4\a + 2)t -(4\a^2 + 6\a + 6)t^2 -(6\a^2 + 4\a + 8)t^3
+(2\a^2 - 4\a - 3)t^4 \\
& \,\,\,\, +(-2\a^2 + 4\a + 3)t^5 +(6\a^2 + 4\a + 8)t^6 + (4\a^2 + 6\a + 6)t^7
+(2\a^2 + 4\a + 2)t^8 +t^9 \,.
\end{align*}
\end{small}

Using the numerical value of $\alpha^\geom$ from~\eqref{m036alpha}, 
this matches (after multiplication with $-1$, and after renaming the `t'
variable by `a') with the $\BC^2$-torsion polynomial of \texttt{m036} computed
in \texttt{SnapPy} 

\begin{lstlisting}
sage:snappy.Manifold('m036').hyperbolic_SLN_torsion(3,bits_prec=50)
-1.0000000000003*a^9 + (2.3829757679074 - 1.0177035576648*I)*a^8 + (1.2222625231179 + 0.19487790074776*I)*a^7 + (-1.0258287470449 + 5.8680293913261*I)*a^6 + (-1.2082197171469 - 7.9034365066556*I)*a^5 + (1.2082197171534 + 7.9034365066503*I)*a^4 + (1.0258287470263 - 5.8680293913033*I)*a^3 + (-1.2222625230950 - 0.19487790079827*I)*a^2 + (-2.3829757679132 + 1.0177035577685*I)*a + 0.99999999993775
\end{lstlisting}


Before we move on to discuss the $\BC^2$-torsion polynomial, note that $M$
has 4 boundary obstruction classes and we presented the Ptolemy and face equations
for only two of them. One can analyze the remaining two similarly, but we will not
give the details here. 

\subsection{$\BC^2$-torsion polynomial}

We finally discuss the $\BC^2$-torsion polynomial.
As shown in Figure~\ref{fig.m036}, the initial surface has six super-Ptolemy
variables $\theta_0,\ldots,\theta_5$. Attaching four tetrahedra to the surface,
we obtain eight additional variables $\theta_6, \ldots, \theta_{13}$ with eight
additional face equations:
\be
\label{eqn.face036}
\begin{aligned}
E_1 : \quad\quad & c_2 \theta_0 +c_9 \theta_7 + c_7 \theta_5 &=0 \\
E_2 : \quad\quad & -c_4 \theta_0  + c_6 \theta_5 - c_9 \theta_6 &=0 \\
E_3 : \quad\quad & -c_1 \theta_3 +c_{10} \theta_9 + c_4 \theta_4 &=0 \\
E_4 : \quad\quad & c_0 \theta_3  + c_3 \theta_4 - c_{10} \theta_8 &=0 \\
E_5 : \quad\quad & c_2 \theta_6 - c_{11} \theta_{11} - c_3 \theta_1 &=0 \\
E_6 : \quad\quad & c_9 \theta_6  - c_0 \theta_4 + c_{11} \theta_{10} &=0 \\
E_7 : \quad\quad & -c_1 \theta_8 - c_{12} \theta_{13} - c_9 \theta_7 &=0 \\
E_8 : \quad\quad & c_{10} \theta_8  + c_6 \theta_7 + c_{12} \theta_{12} &=0
\end{aligned} 
\ee
Solving these equations determines the additional variables
$\theta_6,\ldots,\theta_{13}$ as
\begin{small}
\begin{equation*}
\begin{aligned}
(\theta_6,\ldots,\theta_{13}) &=& \left(  \frac{c_6 \theta_5 - c_4 \theta_0}{c_9},
-\frac{c_7 \theta_5 + c_2 \theta_0}{c_9}, \frac{c_3 \theta_4 + c_0 \theta_3}{c_{10}},
 \frac{-c_4 \theta_4 + c_1 \theta_3}{c_{10}}, 
\frac{c_0 \theta_1 - c_6 \theta_5 + c_4 \theta_0}{c_{11}}, \right. \\
&& -\frac{c_3 c_9 \theta_1 - c_2 c_6 \theta_5 + c_2 c_4 \theta_0}{c_{11} c_{9}},
\frac{c_6 c_7 \theta_5 + c_2 c_6 \theta_0 - c_3 c_9 \theta_4
- c_0 c_9 \theta_3}{c_{12} c_{9}},\\
& & \left.\frac{c_{10} c_7 \theta_5 + c_{10} c_2 \theta_0 - c_1 c_3
\theta_4 - c_0 c_1 \theta_3}{c_{10} c_{12}} \right) \, .
\end{aligned}
\end{equation*}
\end{small}
If we denote by $(\theta'_0,\ldots,\theta'_5)$ the super-Ptolemy variables on the
terminal surface
\be
\label{eqn.idt}
(\theta'_0, \theta'_1, \theta'_2, \theta'_3, \theta'_4, \theta'_5)=
(\theta_{9}, \theta_{12}, \theta_{10}, \theta_{2}, \theta_{13}, \theta_{11})
\ee
so that solving $\theta'_i = \theta_i$ for $0 \leq i\leq 5$ gives a super-Ptolemy
 assignment on $\calT_\varphi$, then Theorem~\ref{thm.jac2} with 
 a Ptolemy assignment given in~\eqref{eqn.csol} gives
 \begin{align*}
 \tau_{2,\rho_c}(t)&=\det \left( t I - \dfrac{\partial \theta'_{i}}
 {\partial \theta_j}\right)= 1 -2t+t^2 +t^4-2t^5 +t^6\,.
 \end{align*}
If we use the sign-deformation given in Section~\ref{sub.m036_3}, the face 
equations~\eqref{eqn.face036} change to
\be
\label{eqn.face036sign}
\begin{aligned}
E_1 : \quad\quad & -c_2 \theta_0 +c_9 \theta_7 - c_7 \theta_5 &=0 \\
E_2 : \quad\quad & +c_4 \theta_0  + c_6 \theta_5 - c_9 \theta_6 &=0 \\
E_3 : \quad\quad & -c_1 \theta_3 +c_{10} \theta_9 - c_4 \theta_4 &=0 \\
E_4 : \quad\quad & c_0 \theta_3  + c_3 \theta_4 - c_{10} \theta_8 &=0 \\
E_5 : \quad\quad & c_2 \theta_6 - c_{11} \theta_{11} + c_3 \theta_1 &=0 \\
E_6 : \quad\quad & -c_9 \theta_6  + c_0 \theta_4 + c_{11} \theta_{10} &=0 \\
E_7 : \quad\quad & -c_1 \theta_8 - c_{12} \theta_{13} - c_9 \theta_7 &=0 \\
E_8 : \quad\quad & c_{10} \theta_8  + c_6 \theta_7 + c_{12} \theta_{12} &=0
\end{aligned} 
\ee
which give
\begin{small}
\begin{equation*}
\begin{aligned}
(\theta_6,\ldots,\theta_{13}) &=& \left(  \frac{c_6 \theta_5 + c_4 \theta_0}{c_9},
\frac{c_7 \theta_5+ c_2 \theta_0}{c_9}, \frac{c_3 \theta_4 + c_0 \theta_3}{c_{10}},
 \frac{c_4 \theta_4 + c_1 \theta_3}{c_{10}}, 
 \frac{-c_0 \theta_1 + c_6 \theta_5 + c_4 \theta_0}{c_{11}}, \right. \\
& & \frac{c_3 c_9 \theta_1 + c_2 c_6 \theta_5 + c_2 c_4 \theta_0}{c_{11} c_{9}},
-\frac{c_6 c_7 \theta_5 + c_2 c_6 \theta_0 + c_3 c_9 \theta_4 + c_0 c_9
\theta_3}{c_{12} c_{9}},\\
& & \left. -\frac{c_{10} c_7 \theta_5 + c_{10} c_2 \theta_0 + c_1 c_3 \theta_4
+ c_0 c_1 \theta_3}{c_{10} c_{12}} \right) \, .
\end{aligned}
\end{equation*}
\end{small}
Applying Theorem~\ref{thm.jac2} with a Ptolemy assignment given
in~\eqref{eqn.csolsign}, we obtain 
\begin{small}
\begin{align*}
\tau_{2,\rho_c}(t)&=
1 + (2 \alpha^2 + 2 \alpha + 2) t + (\alpha^2 + 2 \alpha + 4) t^2
+ (2 \alpha^2 + 4 \alpha + 2) t^3 + (\alpha^2 + 2 \alpha + 4) t^4
\\ & \,\,\,\,
+ (2 \alpha^2 + 2 \alpha + 2) t^5 + t^6
\end{align*}
\end{small}
where $\a$ is a solution to Equation~\eqref{m036alpha}.
%
%
Using the numerical value of $\alpha^\geom$ from~\eqref{m036alpha}, this matches
(after renaming the `t' variable by `a') with the $\BC^2$-torsion polynomial
of \texttt{m036} computed in \texttt{SnapPy} 

\begin{lstlisting}
sage:snappy.Manifold('m036').hyperbolic_SLN_torsion(2,bits_prec=50)
a^6 + (-0.83928675521416 - 1.2125814584144*I)*a^5 + (1.8085121160469 + 0.50885177883279*I)*a^4 + (-2.3829757679062 + 1.0177035576654*I)*a^3 + (1.8085121160468 + 0.50885177883267*I)*a^2 + (-0.83928675521427 - 1.2125814584143*I)*a + 1.0000000000002
\end{lstlisting}

\section{Evidence the conjecture holds for nonfibered manifolds}
\label{sec.evidence}

In this section, we discuss computations showing that
Conjecture~\ref{conj.23} holds for at least one triangulation of 6,672
distinct nonfibered hyperbolic 3-manifolds, of which 14.4\% are
exteriors of knots in $S^3$.  Combined with Theorem~\ref{thm.1}, this
gives compelling evidence for Conjecture~\ref{conj.23}.  For all the
computations discussed here, both the source code and the raw data can
be found at \cite{code:data}.

\subsection{Sample manifolds} 

Our initial list of manifolds was drawn from two sources. The first is
Burton's census of cusped orientable hyperbolic 3-manifolds that can
be triangulated with at most 9 ideal tetrahedra \cite{Burton:census},
where we considered the 59,068 such manifolds where $b_1(M) =
1$. (This census also contains 2,843 manifolds with $b_1(M) > 1$.)
The second source was 1,692 hyperbolic knots in $S^3$, each with less
than 14 crossings, whose exteriors had an ideal triangulation with at
most 16 tetrahedra.  These overlap in 189 manifolds, so full sample
size is 60,571 distinct manifolds. These manifolds were chosen because
each one has a triangulation where Goerner found exact
representations for all possible Ptolemy assignments
\cite{Goerner:unhyp}.  Specifically, for each such assignment he
specifies a number field $K$ as $\BQ[x]/(f(x))$, where
$f(x) \in \BZ[x]$ is irreducible, and the assignment
$c \maps \calT^1 \to K^*$ by a polynomial in $x$ for each edge.

\subsection{Thurston norm and fibering}

\begin{figure}
  \begin{tikzpicture}[nmdstd]
  \begin{tikzoverlay*}[width=0.8\textwidth]{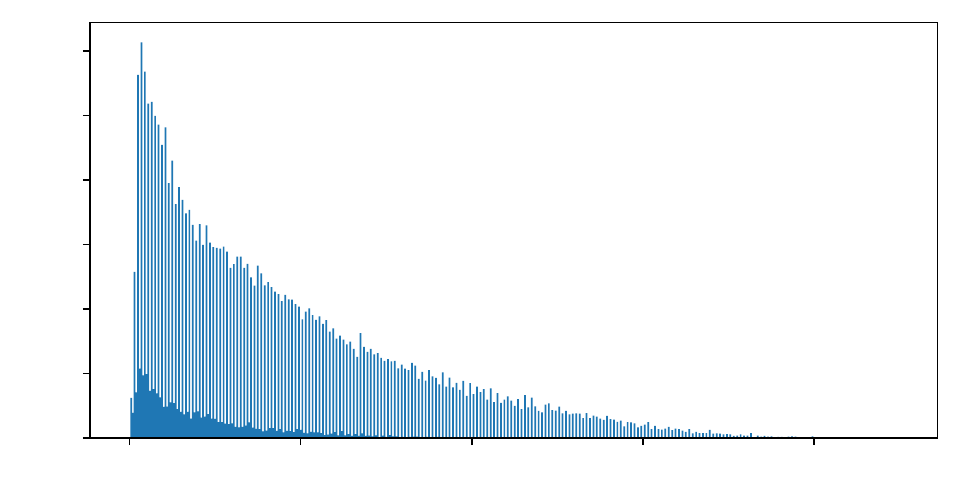}
  \draw (13.510066, 4.531250) node[below] {$0$};

  \draw (31.337647, 4.531250) node[below] {$100$};

  \draw (49.165229, 4.531250) node[below] {$200$};

  \draw (50, 0) node[below]{\small $x(M) = - \chi(\mbox{fiber})$};
  \draw (1, 29) node[left, rotate=90, anchor=south]{\small count};

  \draw (66.992810, 4.531250) node[below] {$300$};

  \draw (84.820391, 4.531250) node[below] {$400$};

  \draw (7.890625, 6.050347) node[left] {$0$};

  \draw (7.890625, 12.751694) node[left] {$200$};

  \draw (7.890625, 19.453041) node[left] {$400$};

  \draw (7.890625, 26.154389) node[left] {$600$};

  \draw (7.890625, 32.855736) node[left] {$800$};

  \draw (7.890625, 39.557083) node[left] {$1000$};

  \draw (7.890625, 46.258430) node[left] {$1200$};

  \begin{scope}[shift={(13.51006594, 6.05034722)},
                xscale=0.17827581, yscale=0.03350674]
  \end{scope}
  \end{tikzoverlay*}
\end{tikzpicture}
  \caption{A histogram of the Thurston norm $x(M)$ for the 53,896
    fibered manifolds in our sample, with mean 86.1, median 65, and
    max 447.  Here, each bar corresponds to the number of manifolds
    with a single value of $x(M)$, e.g.~the tallest bar represents the
    1,227 manifolds where $x(M) = 7$ (for 89.8\% of these, the fiber
    is genus 4 with one puncture).  In this sample, it is much more
    common for $x(M)$ to be odd (93.0\%) than even (7.0\%), and so the
    bars alternate between tall and short.  Compare with
    Figure~\ref{fig.nonfibnorm}.}
  \label{fig.fibnorm}
\end{figure}
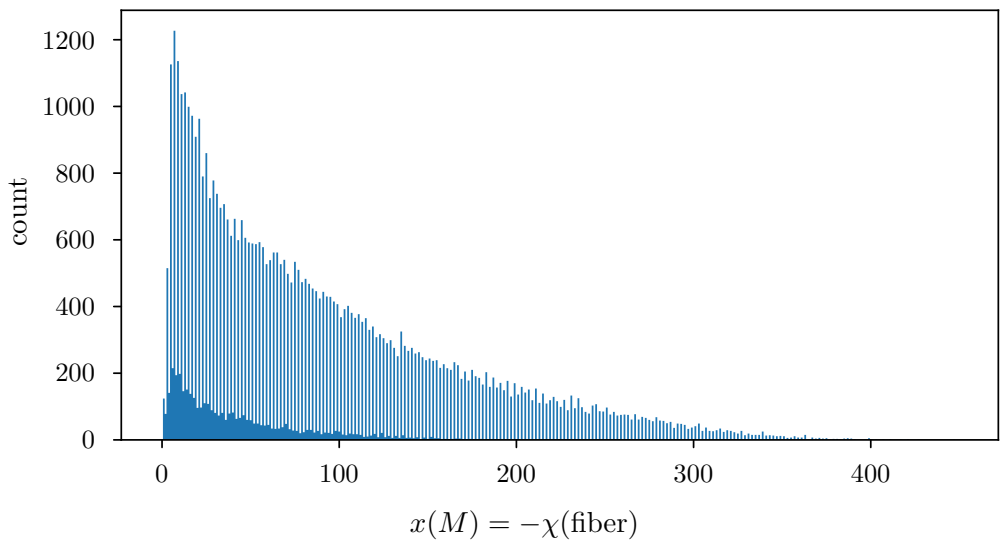

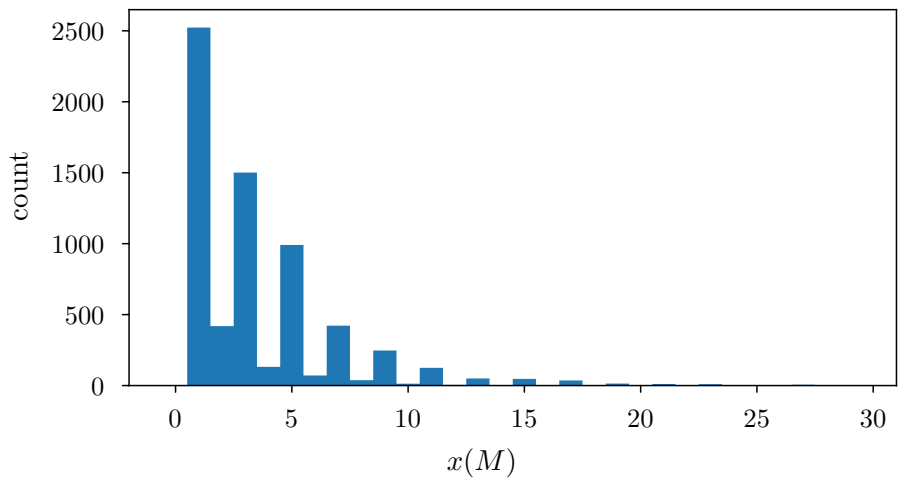
\begin{figure}
  \begin{tikzpicture}[nmdstd]
  \begin{tikzoverlay*}[width=0.7\textwidth]{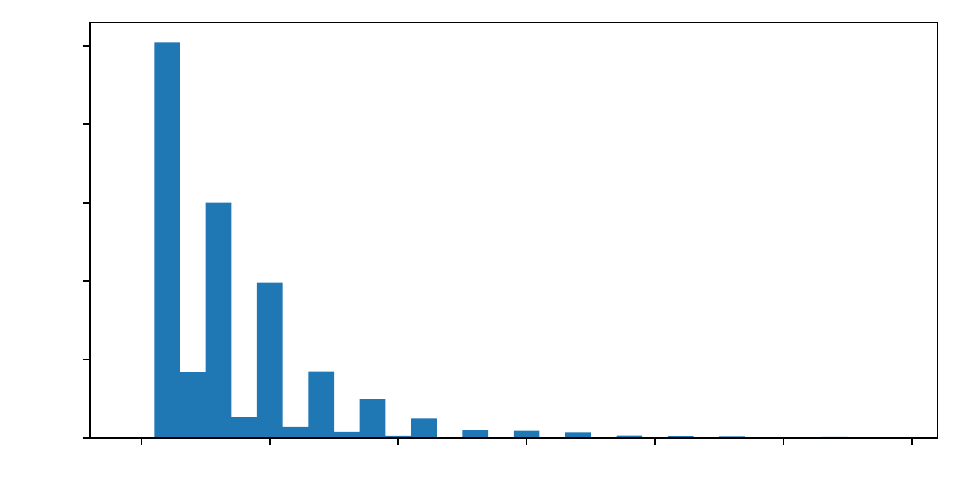}
  \draw (14.733533, 4.427083) node[below] {$0$};

  \draw (28.108165, 4.427083) node[below] {$5$};

  \draw (41.482797, 4.427083) node[below] {$10$};

  \draw (54.857428, 4.427083) node[below] {$15$};

  \draw (68.232060, 4.427083) node[below] {$20$};

  \draw (81.606692, 4.427083) node[below] {$25$};

  \draw (94.981324, 4.427083) node[below] {$30$};

  \draw (7.864583, 5.946181) node[left] {$0$};

  \draw (7.864583, 14.113436) node[left] {$500$};

  \draw (7.864583, 22.280691) node[left] {$1000$};

  \draw (7.864583, 30.447947) node[left] {$1500$};

  \draw (7.864583, 38.615202) node[left] {$2000$};

  \draw (7.864583, 46.782458) node[left] {$2500$};

  \draw (50, 0) node[below]{\small $x(M)$};
  \draw (-1, 29) node[left, rotate=90, anchor=south]{\small count};
  \begin{scope}[shift={(14.73353325, 5.94618056)},
                xscale=2.67492635, yscale=0.01633451]
  \end{scope}
  \end{tikzoverlay*}
\end{tikzpicture}

  \vspace{-0.5em}
  
  \caption{A histogram of the Thurston norm $x(M)$ for the 6,675
    nonfibered manifolds in our sample, with mean 3.6, median 2, and
    max 29.  Compare Figure~\ref{fig.fibnorm}, where the mean
    $x(M)$ is 23.9 times larger.  Despite this, the fibered and
    nonfibered manifolds in the sample have very similar volumes and
    numbers of tetrahedra; see also Figure~\ref{fig.comp}.}
  \label{fig.nonfibnorm}
\end{figure}

\begin{figure}
  \begin{tikzpicture}[nmdstd]
  \begin{tikzoverlay*}[width=0.85\textwidth]{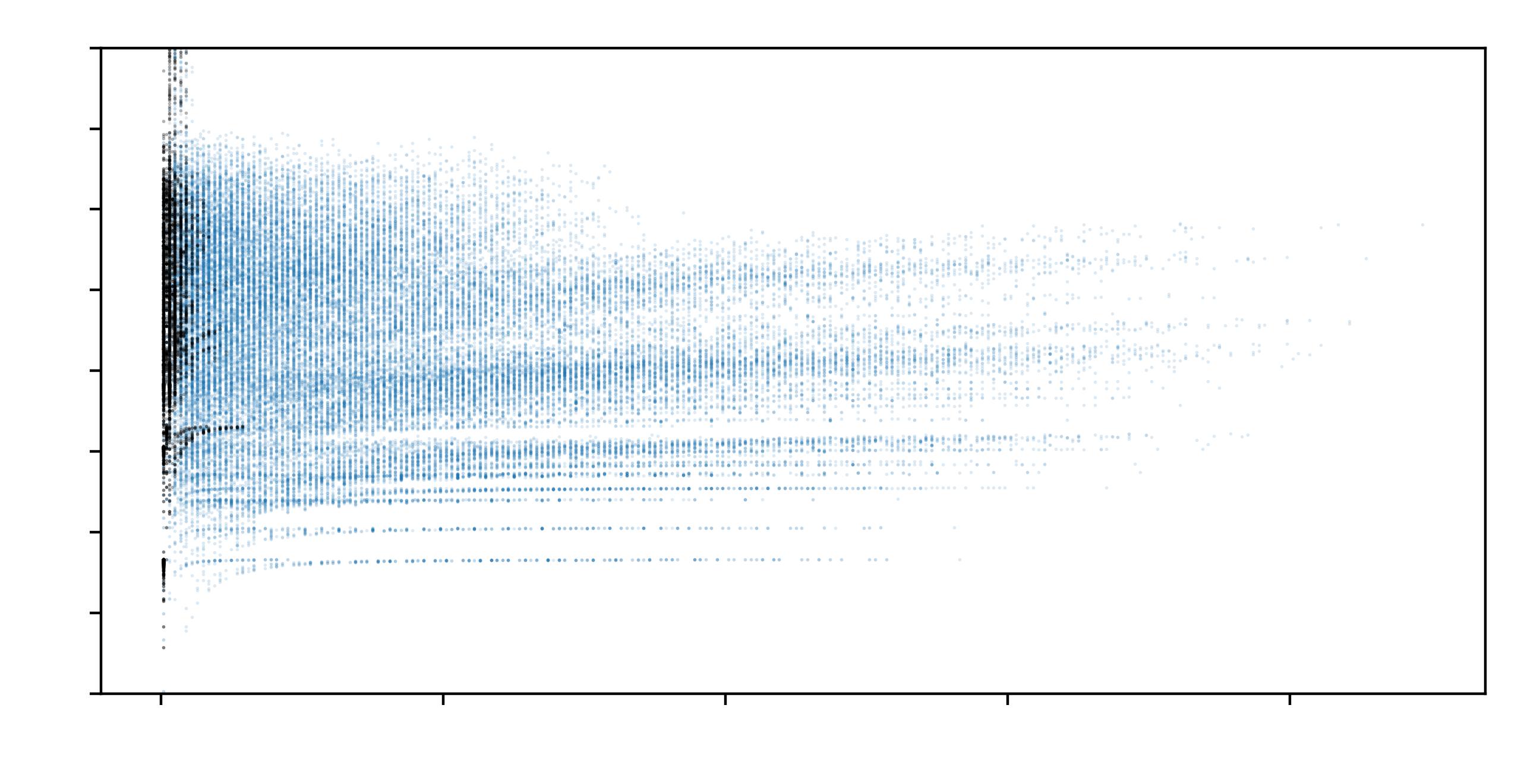}
  \draw (10.575572, 4.427083) node[below] {$0$};

  \draw (29.131012, 4.427083) node[below] {$100$};

  \draw (47.686451, 4.427083) node[below] {$200$};

  \draw (66.241891, 4.427083) node[below] {$300$};

  \draw (84.797330, 4.427083) node[below] {$400$};

  \draw (5.104167, 5.946181) node[left] {$2$};

  \draw (5.104167, 11.254340) node[left] {$3$};

  \draw (5.104167, 16.562500) node[left] {$4$};

  \draw (5.104167, 21.870660) node[left] {$5$};

  \draw (5.104167, 27.178819) node[left] {$6$};

  \draw (5.104167, 32.486979) node[left] {$7$};

  \draw (5.104167, 37.795139) node[left] {$8$};

  \draw (5.104167, 43.103299) node[left] {$9$};

  \draw (5.104167, 48.411458) node[left] {$10$};

  \draw (48, 0) node[below]{\small $x(M)$};
  \draw (1, 27) node[left]{\small $\vol(M)$};
  
  \begin{scope}[shift={(10.57557250, -4.67013889)},
                xscale=0.18555439, yscale=5.30815972]
  \end{scope}
  \end{tikzoverlay*}
\end{tikzpicture}

  \vspace{-0.5em}
  
  \caption{A scatter plot of the Thurston norm $x(M)$ as compared to
    the hyperbolic volume $\vol(M)$ for the 60,571 manifolds in the
    sample.  Here, fibered manifolds are shown in blue and nonfibered
    manifolds in black.  Some 1,222 manifolds with $\vol(M) > 10$ are
    not shown (2\% of the sample), all of which have $x(M) \leq 11$
    and are exteriors of knots in $S^3$.  Some of the structure, e.g.~the
    large number of fibered manifolds with volume roughly 3.66, comes
    from hyperbolic Dehn surgery.}
  \label{fig.comp}
\end{figure}

All these manifolds have $b_1(M) = 1$, and for such we define
$x(M) \in \BZ_{\geq 0}$ to be the Thurston norm of a generator of
$H^1(M; \BZ)$.  For all our sample, we were able to calculate
$x(M)$ using the following technique based on ideas of Lackenby
\cite{Lackenby}.  For each manifold, triangulations were generated
randomly until one was found with a co-orientable taut structure whose
horizontal branched surface carried a non-empty surface (not
necessarily with full support).  Taking said surface $F$ to be
connected, its class generates $H_2(M, \partial M; \BZ)$ and, by
\cite{Lackenby}, it realizes the Thurston norm of $[F]$.

We also determined which of these manifolds fiber over the circle:
53,896 (89.0\%) do and 6,675 (11.0\%) do not.  (For 6,276 (10.4\%)
of these, this was previously done in \cite{Button05} or
\cite{Dunfield:twisted}.) First, we showed that 53,896 of these
manifolds fiber over the circle as follows.  In all but six cases, we
were able to find a co-orientable taut structure whose horizontal
branched surface carries a surface with \emph{full support}.
Equivalently, we found a taut layered ideal triangulation for the
manifold as in Section~\ref{sec.layered}. The other six fibered
manifolds, which were the census manifolds $o9_{33772}$, $o9_{39015}$,
$o9_{39073}$, $o9_{40271}$, $o9_{41509}$, and $o9_{43580}$, were
handled by ad hoc methods that are detailed in \cite{code:data}.  The
remaining 6,675 manifolds were shown to be nonfibered using (a) that
the ordinary Alexander polynomial was not monic (6,463 manifolds), or
(b) that some (exactly computed) twisted Alexander polynomial was not
monic (212 manifolds).
   
Of the 6,675 distinct nonfibered manifolds, exactly 964 (14.5\%) of
them are exteriors of knots in the 3-sphere by
\cite{Dunfield:exceptional}.

\subsection{Checking that 1-loop equals torsion}

\begin{figure}
  \begin{tikzpicture}[nmdstd]
  \begin{tikzoverlay*}[width=0.8\textwidth]{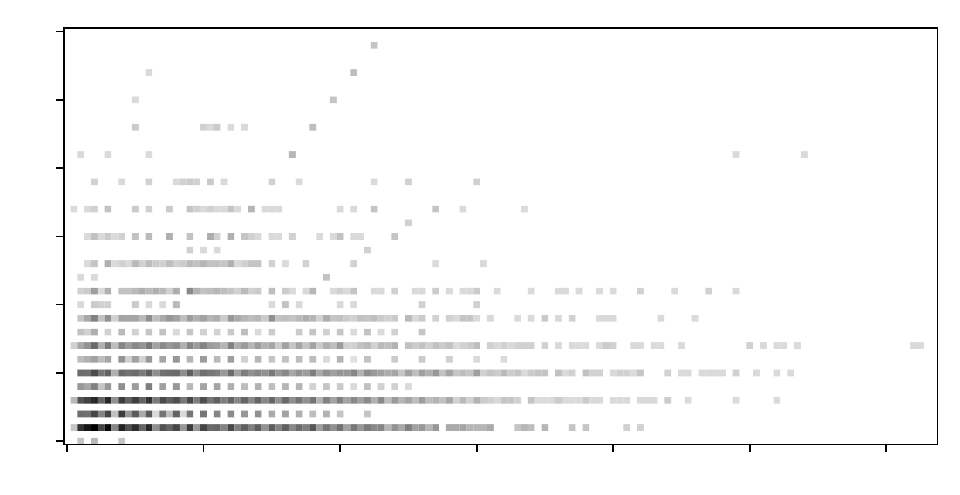}
  \draw (6.991831, 3.734717) node[below] {$0$};

  \draw (21.213701, 3.734717) node[below] {$20$};

  \draw (35.435571, 3.734717) node[below] {$40$};

  \draw (49.657440, 3.734717) node[below] {$60$};

  \draw (50, 0) node[below] {\small $[K : \BQ]$};
  \draw (0, 27) node[left] {\small $\deg \tau_{2, \rho_c}(t)$};

  \draw (63.879310, 3.734717) node[below] {$80$};

  \draw (78.101179, 3.734717) node[below] {$100$};

  \draw (92.323049, 3.734717) node[below] {$120$};

  \draw (5.117188, 5.609361) node[left] {$0$};

  \draw (5.117188, 12.720296) node[left] {$10$};

  \draw (5.117188, 19.831231) node[left] {$20$};

  \draw (5.117188, 26.942166) node[left] {$30$};

  \draw (5.117188, 34.053101) node[left] {$40$};

  \draw (5.117188, 41.164035) node[left] {$50$};

  \draw (5.117188, 48.274970) node[left] {$60$};

  \begin{scope}[shift={(6.99183146, 5.60936144)},
                xscale=0.71109348, yscale=0.71109348]
  \end{scope}
  \end{tikzoverlay*}
\end{tikzpicture}
  \caption{A density plot exploring the 29,948 Ptolemy assignments
    $c \maps \calT^1 \to K^*$ of 6,672 nonfibered hyperbolic
    3-manifolds and their $\BC^2$-torsion polynomials where we checked
    Conjecture~\ref{conj.23}.  Here, $K$ is viewed as an abstract
    number field so that $c$ describes a $\Gal(\Qbar/\BQ)$-orbit of
    such assignments; these encode more than 500,000 distinct concrete
    Ptolemy assignments $c \maps \calT^1 \to \BC^*$. Note here that
    $\deg \tau_{2, \rho_c}(t)$ is always even because of its symmetry
    under $t \mapsto t^{-1}$.}
  \label{fig.degrees}
\end{figure}

For each of the 6,675 nonfibered manifolds, we tried to check
Conjecture~\ref{conj.23} for the preferred triangulation $\calT$ and
all Ptolemy assignments which are part of $0$-dimensional components
of the reduced Ptolemy variety
$P_2(\calT)_\reduced := P_2(\calT)/(\BC^*)^b$. (Any assignment
associated to the hyperbolic structure is on such a component, and
only 162 of these triangulations have any components of positive
dimension.)  We succeeded for all but three manifolds where the
computation did not finish due to the extreme complexity (the
exceptions were $o9_{18365}$, $o9_{20926}$, and $o9_{31289}$).  Recall
that $P_2(\calT)_\reduced$ is defined over $\BQ$, and hence each
$0$-dimensional irreducible component over $\BC$ (i.e.~each isolated
point) necessarily has coordinates in $\Qbar$.  The group
$\Gal(\Qbar/\BQ)$ acts on such points, with the orbits being the
$0$-dimensional $\BQ$-irreducible components.  In particular, each
such orbit of size $d$ has an associated number field $K$ of degree
$d$ with a Ptolemy assignment $c \maps \calT^1 \to K^*$ which gives
all points on this orbit by considering the $d$ distinct embeddings
$K \hookrightarrow \BC$.  As mentioned, the complete list of such
Ptolemy assignments for these manifolds was computed by Goerner
\cite{Goerner:unhyp}.

For each manifold $M$, and boundary obstruction class $\sigma$, and
sign-deformed Ptolemy solution $c \maps \calT^1 \to K^*$ on such a
component, we checked Conjecture~\ref{conj.23} by exact arithmetic in
the number field $K$. There was a mean of 4.5 such $(\sigma, c)$ pairs
per manifold, where $[K : \BQ]$ had median 15, mean 18.8, and a
maximum of 125. (There are four possible
$\sigma \in H^1(\partial M; \BZ/2\BZ)$, two of which are compatible
with Ptolemy assignments associated to the hyperbolic structure.  Each
of the triangulations used supports the hyperbolic structure, so there
were always at least two such pairs $(\sigma, c)$.) The resulting
$\BC^2$-torsion polynomials had median degree 6, mean degree 7.0, and
maximum degree 58, see Figure~\ref{fig.degrees}.  For the
$\BC^3$-torsion polynomial, the degree had median 9, mean 10.4, and
maximum 87.

\begin{remark}
  To validate our code for checking Conjecture~\ref{conj.23}, we also
  ran it on 18,588 of the \emph{fibered} manifolds in our sample. The
  number of pairs $(\sigma, c)$ was similar at 4.1, but $[K : \BQ]$
  was larger, with median 26, mean 28.9, and max 140.
\end{remark}

\subsection{Lower bounds on the Thurston norm}

For a knot $K$ in $S^3$ with exterior $E_K = S^3 \setminus \nu(K)$,
set $x(K) = x(E_K)$; equivalently
\[
x(K) = - \chi(\mbox{minimal Seifert surface}) = 2 \cdot
\mathrm{genus}(K) - 1.
\]
As mentioned in Section~\ref{sec.intro}, if $K$ is a hyperbolic knot
in $S^3$ then for the $\BC^2$-torsion one has:
\[
  x(K) \geq \frac{1}{2} \deg \tau_{2, \rho^{\geom}}(t)
\]
where $\rho^\geom$ is any lift of the holonomy representation to
$\SL$.  This inequality is an equality for all such knots with at most
15 crossings, of which there are more than 300,000, leading to Conjecture~1.7 of
\cite{Dunfield:twisted}:

\begin{conjecture}[Dunfield-Friedl-Jackson]
  \label{conj.DFK}
  If $K$ is a hyperbolic knot in $S^3$, then
  $x(K) = \frac{1}{2} \deg \tau_{2, \rho^{\geom}}(t)$.  Moreover, its
  exterior fibers over the circle if and only if
  $\tau_{2, \rho^{\geom}}(t)$ is monic.
\end{conjecture}
\noindent
While there is an analogous bound on the Thurston norm from the degree
of the $\BC^3$-torsion polynomial, this it is not always sharp
\cite[Section~6.6]{Dunfield:twisted}.

It is natural to ask what happens for more general $M$.  First, the
lower bound $x(M) \geq \frac{1}{2} \deg \tau_{2, \rho^{\geom}}(t)$
holds for any cusped hyperbolic 3-manifold with $b_1(M) = 1$, and
$\tau_{2, \rho^{\geom}}$ is monic whenever $M$ fibers over the circle;
see e.g.~Theorem 1.5 of \cite{Dunfield:twisted}, where the assumption
on the $\BZ/2\BZ$-homology of $M$ is unnecessary.  However, a
new wrinkle is that $\tau_{2, \rho^{\geom}}(t)$ can change
dramatically depending on which lift $\rho^\geom$ is chosen; this was
noted in \cite[Remark~4.7]{Dunfield:twisted} and also illustrated by
Examples~\ref{ex.drama1} and~\ref{ex.drama2} below.

Recall that if a representation $\rhobar \maps \pi_1(M) \to \PSL$
lifts to $\rho \maps \pi_1(M) \to \SL$ there are $H^1(M; \BZ/2\BZ)$
distinct lifts, where the latter is viewed as homomorphisms from
$\pi_1(M)$ to the center $\{\pm I\}$ of $\SL$.  When
$H^1(M; \BZ/2\BZ) \cong \BZ/2\BZ$, for example for the exterior of a
knot in $S^3$, then the two lifts $\rho$ and $\psi$ of the holonomy
representation satisfy $\tau_{2, \rho}(t) = \tau_{2, \psi}(-t)$ by
\cite[Remark~4.6]{Dunfield:twisted} and so contain equivalent
information. To understand the general
situation, we need to study the map
$H_1(\partial M; \BZ) \to H_1(M; \BZ)$ in minute detail.

\subsection{Labeling the boundary obstruction classes}

\begin{figure}
  \centering
  \begin{tikzpicture}[nmdstd]
    \begin{tikzoverlay*}[scale=0.8]{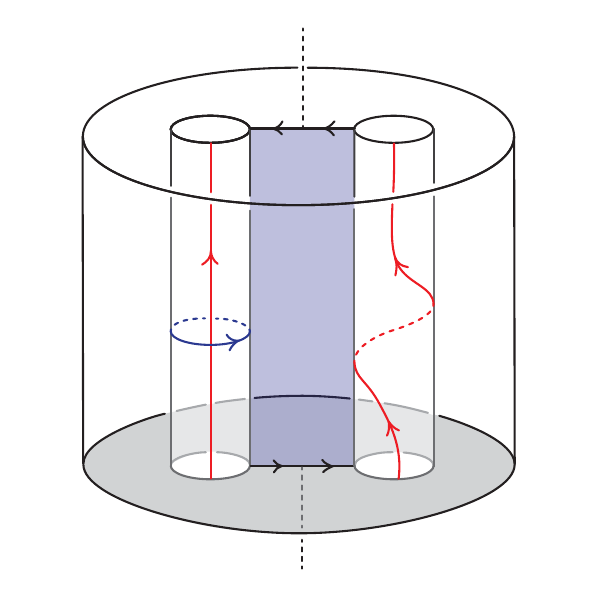}
      \node[] at (50.3,51.4) {$S$};
      \node[] at (22.4,23.2) {$F$};
      \node[below] at (38.9,42.5) {$\beta$};
      \node[left] at (35.1,57.3) {$\alpha$};
      \node[below left] at (66.3,56.2) {$\alpha$};
    \end{tikzoverlay*}
  \end{tikzpicture}
  \caption{ A simple 3-manifold that is $\BZ/2\BZ$-askew.  Let $F$ be an
    annulus and consider $F \times [0, 1]$ as shown above where each
    copy of $F$ includes the rest of the plane it lies in plus a point
    at infinity.  Let $M$ be the mapping torus of the
    self-homeomorphism of $F$ that interchanges the two boundary
    components and looks like rotation by $\pi$ about the indicated
    vertical axis.  Note that any copy of $F$ generates
    $H_2(M; \BZ) \cong H^1(M; \BZ) \cong \BZ$. The vertical strip
    shown becomes a M\"obius band $S$ in $M$ with
    $H_2(M; \BZ/2\BZ) \cong H^1(M; \BZ/2\BZ) \cong (\BZ/2\BZ)^2$
    generated by $[F]$ and $[S]$. Here, the curve $\beta$ is a
    $\BQ$-longitude as $2 \beta = \partial F$, but not a
    $\BZ/2\BZ$-longitude as it meets $S$ once.}
  \label{fig.askew}
\end{figure}
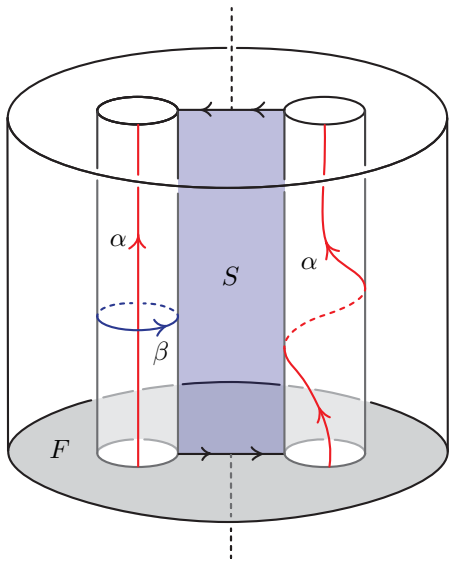

Let $M$ be an orientable 3-manifold where $\partial M$ is a torus with
no restriction on $b_1(M)$.  For a field $\BF$, an
\emph{$\BF$-longitude} is a \emph{primitive} element of
$H_1(\partial M; \BZ)$ that becomes $0$ in $H_1(M; \BF)$; these always
exist as the kernel of
$i_* \maps H_1(\partial M; \BF) \to H_1(M; \BF)$ is 1-dimensional.  A
$\BQ$-longitude is the same as either of the generators of the kernel
of
\begin{equation}
  \label{eq.istar}
  i_* \maps H_1(\partial M; \BZ) \to H_1(M; \BZ)/\mathrm{torsion}
\end{equation}
A $\BQ$-longitude need not be a $\BZ/2\BZ$-longitude, and in this
situation we say $M$ is \emph{$\BZ/2\BZ$-askew}; see
Figure~\ref{fig.askew} for the example of the twisted $I$-bundle over
the Klein bottle.  If instead each $\BQ$-longitude is a
$\BZ/2\BZ$-longitude, we say that $M$ is \emph{$\BZ/2\BZ$-aligned}.
For example, the exterior of any knot in $S^3$ is $\BZ/2\BZ$-aligned.
Note that $M$ is $\BZ/2\BZ$-askew if and only if the composition
$H^1(M; \BZ) \to H^1(\partial M; \BZ) \to H^1(\partial M; \BZ/2\BZ)$
is $0$.  In our sample of nonfibered manifolds, 92.3\% are
$\BZ/2\BZ$-aligned and $7.7\%$ are $\BZ/2\BZ$-askew (for the fibered
manifolds the breakdown is $93.8\%$ and $6.2\%$).

We define a \emph{homologically reasonable framing} to be a generating
set $(\alpha, \beta)$ for $H_1(\partial M; \BZ)$ where $\beta$ is a
$\BQ$-longitude and $\alpha$ is \emph{not} a $\BZ/2\BZ$-longitude.
Here, $\alpha$ will generate the image of $i_*$ in (\ref{eq.istar}),
as well as image of
$i_* \maps H_1(\partial M; \BZ/2\BZ) \to H_1(M; \BZ/2\BZ)$.  When $M$
is the exterior of a knot in $S^3$, any standard (meridian, longitude)
basis is homologically reasonable. See Figure~\ref{fig.askew} for such
a basis when $M$ is $\BZ/2\BZ$-askew. Note $\beta$ is unique up to
sign, and the possible $\alpha$ are $\alpha' = \pm \alpha + n \beta$
where $n \in 2\BZ$ if $M$ is {$\BZ/2\BZ$-askew and $n \in \BZ$
  otherwise.
  
If $(\alpha, \beta)$ is a homologically reasonable framing, consider
the map $H^1(\partial M; \BZ/2\BZ) \to (\BZ/2\BZ)^2$ talking $\phi$
to $(\phi(\alpha), \phi(\beta))$.  We label the four elements of
$H^1(\partial M; \BZ/2\BZ)$ lexicographically as
\[
  \sigma_0, \sigma_1, \sigma_2, \sigma_3 = [(0, 0), (0, 1), (1, 0), (1, 1)].
\]
Here, the image of
$i^*: H^1(M; \BZ/2\BZ) \to H^1(\partial M; \BZ/2\BZ)$ is
$\{\sigma_0, \sigma_2\}$ when $M$ is $\BZ/2\BZ$-aligned and
$\{\sigma_0, \sigma_3\}$ when $M$ is $\BZ/2\BZ$-askew.

When $M$ is $\BZ/2\BZ$-aligned, this labeling is not canonical: a
different framing can interchange $\sigma_1$ and $\sigma_3$,
specifically $(\alpha + \beta, \beta)$.  However, when $M$ is
$\BZ/2\BZ$-askew, labeling is \emph{independent} of the homologically
reasonable framing.  (A coordinate-free way to see this is that
$\sigma_0$ is the identity element, $\sigma_3$ is the nonzero element
in the image of $H^1(M; \BZ/2\BZ)$, and $\sigma_2$ is the nonzero
element that vanishes on the $\BQ$-longitude.)

Now suppose $b_1(M) = 1$ and let $\epsilon_\infty$ be the unique
nonzero element in the image of $H^1(M; \BZ) \to H^1(M; \BZ/2\BZ)$.
If $\rho \maps \pi_1(M) \to \SL$, then
$\tau_\rho(t) = \tau_{\epsilon_\infty \cdot \rho} (-t)$ as in
\cite[Remark~4.6]{Dunfield:twisted}.  When $M$ is $\BZ/2\BZ$-aligned,
the restriction of $\epsilon_\infty$ to $\partial M$ is $\sigma_2$.
Thus if $c$ is a Ptolemy solution for the boundary obstruction class
$\sigma_i$, there is a corresponding $c'$ with boundary obstruction
class $\sigma_{i + 2 \pmod 4}$ so that
$\rho_{c'} = \epsilon_\infty \cdot \rho_c$ and hence
$\tau_{\rho_c'}(t) = \tau_{\rho_c} (-t)$.  Thus for $\BZ/2\BZ$-aligned
$M$, to extract all possible Thurston norm information from the
$\tau_{\rho}(t)$ it suffices to consider only the boundary obstruction
classes $\{\sigma_0, \sigma_1\}$.  Note that a lift of the holonomy
representation will correspond to $\sigma_1$ or $\sigma_3$ for a
$\BZ/2\BZ$-aligned $M$ by \cite{Calegari:real}.

In contrast, when $M$ is $\BZ/2\BZ$-askew, the class $\epsilon_\infty$
becomes $0$ in $H^1(\partial M; \BZ/2\BZ)$ and so all four boundary
obstruction classes need to be considered when trying to understand
the Thurston norm.  Here, a lift of the holonomy representation will
correspond to $\sigma_1$ or $\sigma_2$.  We now give two examples that
illustrate a new behavior that can arise when $M$ is $\BZ/2\BZ$-askew:
for the holonomy representation, the degree of $\tau_{2,\rho}(t)$ can depend
on the choice of lift.

\begin{example}
  \label{ex.drama1}
  Let $M$ be the census manifold $t12835$ which is $\BZ/2\BZ$-askew
  and has an ideal triangulation with $8$ tetrahedra.  Here
  $H_1(M) = \BZ \oplus \BZ/6\BZ$ and $x(M) = 2$.

  For the boundary obstruction class $\sigma_1$, we find using
  \cite{code:data} that there are two corresponding lifts of the
  holonomy representation. These are in the same Galois orbit and
  the abstract number field is $K = \BQ(\zeta)$ where the minimal
  polynomial of $\zeta$ is $x^4 - x^2 + 1$; this is the cyclotomic
  field where $\zeta$ is a primitive 12th root of unity.

  The $\BC^2$-torsion polynomial is
  \[
    4 (t^2 + t^{-2})  + (-13\zeta^3 +
    9\zeta)(t + t^{-1}) + (-24 \zeta^2 + 16)
  \]
  which has degree 4 and so gives a sharp bound on $x(M)$ and shows
  that $M$ is not fibered.  (The embeddings of $K$ that give lifts of
  the holonomy representation are $\zeta \mapsto e^{-\pi i/6}$ and
  $\zeta \mapsto -e^{-\pi i /6}$; here, the two representations differ by
  $\epsilon_\infty$.  The other two embeddings give lifts of the
  holonomy representation for the manifold with the opposite
  orientation.)

  For the boundary obstruction class $\sigma_2$, we also find
  two lifts of the holonomy representation (again differing by
  $\epsilon_\infty)$ with the same $K$ (of necessity) and choices of
  $\zeta$.  However, this time the $\BC^2$-torsion polynomial is:
  \[
    (3\zeta^3 + 5\zeta) (t + t^{-1})
  \]
  which has only degree 2 in $t$! In particular, it is \emph{not} the
  case that every lift of the holonomy representation gives a sharp
  bound on $x(M)$. 
\end{example}

\begin{example}
  \label{ex.drama2}
  Let $M$ be the census manifold $o9_{43413}$ which is
  $\BZ/2\BZ$-askew and has an ideal triangulation with $9$ tetrahedra,
  $H_1(M; \BZ) = \BZ \oplus \BZ/2$, and $x(M) = 4$.

  For the boundary obstruction class $\sigma_1$, we find there are two
  corresponding lifts of the holonomy representation.  The field is
  $K = \BQ(\omega)$ where the minimal polynomial of $\omega$ is
  $x^6 - x^2 + 1$, and the two relevant embeddings in are
  $\omega \approx \pm (0.87498455 + 0.32130825 i)$. The $\BC^2$-torsion
  polynomial of these representations has degree 8:
  \[
    4(t^4 + t^{-4}) + (-6\omega^4 + 3\omega^2 - 2)(t^2 + t^{-2}) + (20\omega^4 - 6\omega^2 - 12)
  \]
  which gives a sharp bound on $x(M)$ and shows that $M$ is not
  fibered.  Whereas, for the boundary obstruction class $\sigma_2$, we
  find the $\BC^2$-torsion is
  \[
    (2\omega^4 + \omega^2 - 2)(t^2 + t^{-2})
    + (-4\omega^5 - 4\omega^3 + 2\omega)(t + t^{-1}) + (4\omega^4 + 2\omega^2 - 4)
  \]
  which has degree $4$ in $t$ (the two relevant embeddings in $\BC$
  remain the same).  At the same time, the field generated by the
  coefficients of the polynomial is now the whole trace field, whereas
  for the first polynomial it was an index-2 subfield.
\end{example}

\subsection{Sharpness of  torsion bounds on the Thurston norm}
For our nonfibered manifolds, we find using \cite{code:data} that:
\begin{theorem}
  \label{thm.sharpy}
  For each of the 6,672 nonfibered hyperbolic manifolds above, there
  is at least one lift $\rho$ of the holonomy representation where
  $x(M) = \frac{1}{2} \tau_{2, \rho} (t)$.  Indeed, this is the case
  for all $\rho$ corresponding to the boundary obstruction class
  $\sigma_1$.  In contrast, there are 50 manifolds (necessarily
  $\BZ/2\BZ$-askew) where some lift corresponding to $\sigma_2$ has
  $x(M) > \frac{1}{2} \tau_{2, \rho} (t)$.  Finally, for
  \emph{every} lift of the holonomy representation
  $\tau_{2, \rho} (t)$ is nonmonic.
\end{theorem}
A caveat is that $\tau_{2, \rho} (t)$ is nonmonic is ``barely true''
for $o9_{31518}$, where the leading coefficient is $-1$; recall from
\cite{Dunfield:twisted} that using the sign-refined torsion, the
polynomial $\tau_{2, \rho} (t)$ is defined on the nose.

In contrast, the ordinary Alexander polynomial $\Delta_M$ does not
detect $x(M)$ for 263 (3.9\%) of these manifolds, and is monic for 212
(3.2\%) of them (together, these amount to 362 (5.4\%) manifolds).
For the nonfibered knots looked at in \cite{Dunfield:twisted}, the
analogous numbers were 8,834 (4.5\%) and 7,972 (4.1\%); thus the
``interesting'' portion of our sample is 25-30 times smaller than that
of \cite{Dunfield:twisted}. (There is no theorem that the
$\BC^2$-torsion polynomial has to do at least as well as the basic
$\Delta_M$ in either setting) The subset of $\BZ/2\BZ$-askew manifolds
where $\Delta_M$ does not detect $x(M)$ is just 56 manifolds (these
include all 50 where some lift corresponding to $\sigma_2$ has
$x(M) > \frac{1}{2} \tau_{2, \rho} (t)$).  Given this, we are not so
bold as to propose an analogue of Conjecture~\ref{conj.DFK} for all
1-cusped hyperbolic 3-manifolds with $b_1(M) = 1$.  Another natural
question is to explain why lifts associated to $\sigma_1$ sometimes
outperform those associated to $\sigma_2$; concretely, the former are
those lifts $\rho$ where $\tr(\rho(\beta)) = -2$ for a $\BQ$-longitude
$\beta$.


\bibliographystyle{amsalphaurl}
\def\MR#1{\href{http://www.ams.org/mathscinet-getitem?mr=#1}{MR#1}}
\bibliography{biblio}
\end{document}